\documentclass[10pt]{article}
\usepackage{graphicx}
\usepackage{latexsym}
\usepackage{xcolor}
\usepackage{amsmath,amsthm,amssymb}
\usepackage{mathrsfs}
\usepackage[hypcap]{caption}
\usepackage[all]{xy}
\usepackage[textwidth=14.5cm]{geometry}

\usepackage{tabularx}
\makeatletter
 
 \@addtoreset{equation}{section}
\makeatother
\makeatletter
\newcommand{\subjclass}[2][1991]{%
  \let\@oldtitle\@title%
  \gdef\@title{\@oldtitle\footnotetext{#1 \emph{Mathematics subject classification.} #2}}%
}
\newcommand{\keywords}[1]{%
  \let\@@oldtitle\@title%
  \gdef\@title{\@@oldtitle\footnotetext{\emph{Key words and phrases.} #1.}}%
}
\makeatother

\theoremstyle{plain}
\newtheorem{df}{\sc \bf Definition}[section]
\newtheorem{thm}[df]{\sc \bf Theorem}

\newtheorem{prop}[df]{\sc \bf Proposition}
\newtheorem{lem}[df]{\sc \bf Lemma}
\newtheorem{cor}[df]{\sc \bf Corollary}
\newtheorem*{conj}{\sc \bf Conjecture}
\theoremstyle{remark}
\newtheorem{rem}[df]{\sc Remark}

\newcommand{\R}{\boldsymbol{R}}

\newcommand{\C}{\boldsymbol{C}}

\newcommand{\m}{\mathfrak{m}}
\newcommand{\g}{\mathfrak{g}}

\newcommand{\da}{\mathfrak{a}}
\newcommand{\db}{\mathfrak{b}}
\newcommand{\sll}{\mathfrak{sl}}

\newcommand{\PL}{PL}
\newcommand{\GL}{GL}

\usepackage{enumerate}
\usepackage{multirow,bigdelim}
\usepackage{delarray}

\begin{document}
\title{Castling transformations of projective structures}


\date{\vspace{-5ex}}
\author{Hironao Kato   
\thanks{The author is supported by JSPS and JSPS Strategic Young Researcher Overseas Visits Program for  Accelerating Brain Circulation.}}


\maketitle
\begin{abstract}
We construct an infinite sequence of projectively flat manifolds by using castling transformations of prehomogeneous vector spaces.   
We also give a classification of manifolds equipped with a flat projective structure obtained by a finite number of castling transformations, and describe these flat projective structures by atlases.  

\ 
\end{abstract}

\vspace{-5mm}
\section{Introduction} 
A flat Grassmannian structure of type $(\beta, \alpha)$ on a manifold $M$ is a maximal atlas $\{(U_a, \varphi_a)\}_{a \in A}$ of $M$ whose charts $\varphi_a$ take values in the Grassmannian manifold 
$Gr_{\alpha, \alpha + \beta}$ and coordinate changes $\varphi_b \circ \varphi_a^{-1}$ belong to the projective linear group $PL(\alpha + \beta) := GL(\alpha + \beta)/R^* I$. 
When $\alpha =1$, this notion gives a definition of flat projective structures on $M$. Obviously the projective spaces admit a flat projective structure. 
The classification of manifolds admitting a flat projective structure is still widely open 
(cf. \cite[chapter 6]{ovsienko-tabachnikov}) and active area. 
Indeed, recently in \cite{goldman-cooper} it has been proved that a connected sum $\R P^3 \# \R P^3$ does not admit a flat projective structure. 
In our last paper \cite{hkato}, we proved that invariant flat complex projective structures on complex Lie groups 
correspond to certain infinitesimal prehomogeneous vector spaces. 

In the theory of prehomogeneous vector spaces there is a notion of castling transformations, 
which is a certain transformation of linear representations of algebraic groups preserving the prehomogeneity.  
In this paper we  establish a transformation of manifolds equipped with 
a projective structure as a generalization of castling transformations.  
As castling transformations preserve the prehomogeneity of representations,  
our castling transformations of projective structures preserve the projectively flatness. 
Moreover, since we can repeat a castling transformation,   
we can construct a sequence of projectively flat manifolds from a given projectively flat manifold. 
In fact we prove the following: 
Let $\{(U_a, \varphi_a)\}_{a \in A}$ be a flat Grassmannian structure of type $(\beta, \alpha)$ on $M$. 
Assume $\alpha + \beta \geq 3$ and $\alpha \leq \beta$. 

\medskip
\begin{thm}\label{cor of thm} 
By a finite number of castling transformations from $\{(U_a, \varphi_a)\}_{a \in A}$ we obtain a projectively flat manifold $N$, which is a principal fiber bundle over $M$.    
There is a one-to-one correspondence between the set of 
structure groups $\prod_{i=1}^j \PL(k_i)$ of $N$  and 
the set of solutions $(k_1, \ldots, k_j)$ of the Grassmannian type equation 
\vspace{-1mm}
\[(\ast) \ \ \alpha \beta + k_1^2 + \cdots + k_j^2 - j  - (\alpha + \beta) k_1 \cdots k_j +1 = 0 \] 
satisfying $k_i \geq \alpha$ \mbox{\rm(}$1 \leq  i \leq j$\mbox{\rm)} and $j \geq 1$. 
\end{thm} 

\medskip
The projectively flat manifold $N$ is described by using atlases in the last section.  
The case $\alpha = 1$ corresponds to the assumption that $M$ admits a flat projective structure. 
Thus from any projectively flat manifold $M$, we can obtain a projectively flat principal fiber bundle $N$ over $M$ with group $\prod_{i=1}^j \PL(k_i)$ satisfying the equation ($\ast$).    
Furthermore the theorem yields  
a sequence of projectively flat manifolds, which are connected by manifolds equipped with a flat Grassmannian structure.  
Each flat projective structure on $N$ is right invariant under the action of 
$\prod_{i=1}^j \PL(k_i)$. 

We note that a flat projective structure exists on $M$ iff a projectively flat affine connection exists on $M$. Thus a flat affine connection induces a flat projective structure.  
However about the existence problem there is the following obstruction: A simply connected compact manifold admitting a flat projective 
structure is diffeomorphic to the sphere $S^n$ (see \cite{kobayashi-nagano}). 
Thus the manifold $\prod_{i=1}^j S^{n_i}$ ($n_i, j \geq 2$) does not admit any flat projective structure. 
This point distinguishes flat projective structures from flat affine connections and flat Riemannian metrics 
as any product of flat affine (resp. Riemannian) manifolds  is a flat affine (resp. Riemannian) manifold again. 
However, in \cite{hkato} we obtained a real Lie algebra $\sll(k_1) \times \cdots \times \sll(k_j)$ with a certain condition whose corresponding real Lie group admits a invariant flat real projective structure.  
Another aim of this paper is to generalize these examples from the view point of Grassmannian structures. 

The paper is organized as follows. 
First of all we review the Grassmannian structures in $\S$ 1 and 
establish a castling transformation of projective structures by using 
Cartan connections in $\S$ 2 and $\S$ 3.  
In $\S$ 4 we investigate the base spaces obtained by successive castling transformations and describe 
a relation between base spaces.  
$\S$ 5 is devoted to some examples of base spaces. 
In $\S$ 6 we investigate the positive integer solutions of the Grassmannian type equation ($\ast$), and give one  conjecture.  In $\S$ 7 we describe flat projective structures constructed in Theorem \ref{cor of thm} by using atlases. 

\section{Preliminaries} 
{\bf 2.1 \ Grassmannian structures and projective structures} 

\medskip
\noindent
Throughout this paper by a manifold we mean a $C^\infty$ real manifold.  
We recall the notion of Grassmannian structures and projective structures to establish castling transformations 
in the differential geometry. 
Let $M$ be a real manifold of dimension $r$. 
Denote by $\mathcal{L}(M)$ a bundle of linear flames of $M$ and we regard an element of $\mathcal{L}(M)$ 
as a linear isomorphism $\R^r \to T_pM$.   
We identify $\R^r$ with $\R^n \otimes \R^m$ and  
consider a $\GL(n) \otimes \GL(m)$-structure $P_tM$, i.e. a subbundle of $\mathcal{L}(M)$ with structure group $\GL(n) \otimes \GL(m)$.  
If we have $n, m \geq 2$, 
we call $P_tM$ a Grassmannian structure of type $(n, m)$ on $M$ in this paper.  
Note that if $n = 1$ or $m = 1$, then $P_tM = \mathcal{L}(M)$. Put $l := m + n$.  
There are various names and definitions.  
In \cite{hangan} and \cite{ishihara} a $\GL(n) \otimes \GL(m)$-structure is called 
a tensor product structure. 
On the other hand in \cite{sato-machida}, an isomorphism  
$\sigma: TM \to V \otimes W$ itself is called a Grassmannian structure, where $V$ and $W$ are vector bundles with rank $n$ and $m$ over $M$ ($n, m \geq 2$).  Such an isomorphism $\sigma$ gives a $GL(n) \otimes GL(m)$-structure in a natural 
manner, however the author does not know whether the converse is true.  
Typical examples admitting a Grassmannian structure are Grassmannian manifolds 
(see \cite{sato-machida} for other examples). Denote by  $Gr_{m, l}$ a Grassmannian 
manifold consisting of $m$-dimensional subspaces in the $l$-dimensional real vector space $W$. 
The real projective transformation group $PL(W)$ acts on $Gr_{m, l}$ transitively. 
Let $\{a_1, \ldots, a_{l}\}$ be a linear basis of $W$. 
With respect to the basis $\{a_i\}_{i=1}^l$ the group $PL(W)$ is expressed as the quotient $GL(l)/\R^* I_{l}$, 
which we denote by $PL(l)$.  
Note that now the basis $\{a_1, \ldots, a_{l}\}$ is identified with the natural basis of $\R^l$. 
Let $v$ be a linear frame $(a_1, \cdots, a_{m})$.  
We denote by $<\!v\!>$ the $m$-dimensional subspace spanned by $v$. 
Let $PL(l)_{<v>}$ be the isotropy  
subgroup at $<\!v\!>$. Then we have $Gr_{m, l} = PL(l)/\PL(l)_{<v>}$. The Lie algebra of $PL(l)$ is isomorphic to 
$\sll(l)$, which has the graded decomposition $\sll(l) = \g_{-1} + \g_0 + \g_{1}$ given by 
\begin{eqnarray*}
\g_{-1} \! \! &=& 
\left\{ 
      \left.
           \begin{pmatrix}
                 0 & 0 \\
                 C & 0 
           \end{pmatrix}  
      \right|
          C \in M(n,m) 
\right\}, \\ 
\g_0 &=&
\left\{
     \left.
           \begin{pmatrix}
                 A & 0 \\
                 0 & B 
           \end{pmatrix}  
      \right|
          \begin{array}{l}
            A \in \mathfrak{gl}(m), \ B \in \mathfrak{gl}(n) \\
            \mathrm{tr}(A + B) = 0 
          \end{array}  
\right\}, \\
\g_1 &=&
\left\{
      \left.
           \begin{pmatrix}
                 0  & D \\
                 0 & 0 
           \end{pmatrix}  
      \right|
         D \in M(m,n) 
\right\}.
\end{eqnarray*}    
The vector space $\R^n \otimes {\R^m}$ is naturally identified with $\g_{-1}$ and  
the isotropy representation $\rho: \PL(l)_{<v>} \to \GL(g_{-1})$  
is given by 
\[\rho: 
\begin{pmatrix}
       A & C \\
       0 & B 
\end{pmatrix}
\mapsto 
B \otimes {}^tA^{-1}.
\] 
Thus the image of $\rho$ is the group  
$\GL(n) \otimes \GL(m)$. 
The isotropy representation $\rho$ enables us to identify 
$GL(n) \otimes GL(m)$ with the subgroup $G_0$ of $\PL(l)$; 
\[G_0 = \left\{\left. \begin{pmatrix} 
                 A & 0 \\ 
                 0 & B 
           \end{pmatrix} \right|  A \in \GL(m), B \in \GL(n)\right\}.\]   
Thus we obtain the imbedding $\imath: GL(n) \otimes GL(m) \to \PL(l)_{<v>}$, which is defined by 
$A \otimes B \mapsto 
\begin{pmatrix}
      {}^t B^{-1} & 0 \\ 
        0 & A 
\end{pmatrix}$. 
Moreover the Lie algebra  
$gl(n)\otimes I_m + I_n \otimes gl(m)$ is identified with $\g_0$. 


Here we recall the notion of ($PL(l)$, $Gr_{m,l}$)-structures on $M$. A ($PL(l)$, $Gr_{m,l}$)-structure on $M$ is a maximal atlas $\{(U_a, \varphi_a)\}_{a \in A}$ of $M$ 
satisfying the following condition (cf. \cite{goldman3}, \cite{hkato}): 

(1) $\{U_a\}_{a \in A}$ is an open covering of $M$,

(2) $\varphi_a$ maps $U_a$ diffeomorphically onto an open subset of $Gr_{m,l}$, 

(3) for every pair 
$(b, a)$ with $U_a \cap U_b \neq \emptyset$ and each connected component 

\quad \ $C$ of $U_a \cap U_b$,  
$\varphi_b \circ \varphi_a^{-1} \! \! \mid_{\varphi_a(C)}$ is given by an element of $PL(l)$. 

\noindent 
We call a $(PL(l), Gr_{m,l})$-structure a flat Grassmannian structure of type $(n, m)$. A flat Grassmannian structure of type $(n, 1)$ is nothing but a flat projective structure. 

Now we introduce the notion of Grassmannian Cartan connections. Let $Q$ be a principal $\PL(l)_{<v>}$-bundle over $M$ and $\omega$ be a $\sll(l)$-valued 1-form on $Q$.  
Then the pair $(Q, \omega)$ is called a Grassmannian Cartan connection of type $(n, m)$ on $M$ if the following conditions are satisfied: 
\smallskip

(1) $\omega$: $T_uQ \to \sll(l)$ gives a linear isomorphism,  

(2) $R_g^*\omega = Ad(g^{-1})\omega$ for $g \in \PL(l)_{<v>}$, 

(3) $\omega(A^*) = A$  ($A \in \sll(l)_{<v>}$), where $A^*$ is the fundamental vector field.  

\medskip 
A $\sll(l)$-valued 2-form $\Omega$ on $Q$ defined by $\Omega = d\omega + \frac{1}{2}[\omega, \omega]$ is 
called a curvature form. A Grassmannian Cartan connection $(Q, \omega)$ is said to be flat if $\Omega = 0$. 
Now we recall there is the following one-to-one correspondence (cf. \cite{hkato}): 
\begin{eqnarray*}
&& \{\mbox{$(PL(l), Gr_{m,l})$-structures on $M$}\}  \\
&& \to \{\mbox{flat Grassmannian Cartan connections of type $(n, m)$ on $M$}\}/_\sim.
\end{eqnarray*} 
The equivalence relation of the latter set denotes the isomorphisms of Cartan connections. 
Generally a Grassmannian Cartan connection $(Q, \omega)$ of type $(n, m)$ over $M$ induces a $\GL(n) \otimes \GL(m)$-structure of $M$ 
as follows (cf. \cite[p.135]{tanaka2}).  
Let $\rho: PL(l)_{<v>} \to GL(g_{-1})$ be the isotropy representation. We denote the kernel of $\rho$ by $ker\rho$. 
Then $PL(l)_{<v>}/{ker\rho} \cong  GL(n) \otimes GL(m)$.  
Thus the quotient manifold $\widetilde{Q} := Q/{ker\rho}$ is regarded as a principal fiber bundle over $M$ with structure group 
$GL(n) \otimes GL(m)$. Let $\omega_{-1}$ (resp. $\omega_{0}$) be the $\g_{-1}$ (resp. $\g_{0}$) component of the 1-form $\omega$.  
By using the natural projection $\rho: Q \to \widetilde{Q}$, we obtain the $\g_{-1}$-valued 1-form $\theta$ 
on $\widetilde{Q}$ defined by $\rho^* \theta = \omega_{-1}$.  Then ($\widetilde{Q}$, $\theta$) can be regarded as 
a $GL(n) \otimes GL(m)$-structure and its canonical form. 
We define an injection 
$\iota: GL(n) \otimes GL(m) \hookrightarrow PL(l)_{<v>}$ by 
\[\iota(A \otimes B) =  
\begin{pmatrix} 
       {}^t B^{-1} & 0 \\
       0 & A 
\end{pmatrix}.\]  
The homogeneous space $PL(l)_{<v>}/\iota(GL(n) \otimes GL(m))$ is homeomorphic to $\g_{1}$, and 
hence there exists a bundle homomorphism $h:\widetilde{Q} \hookrightarrow Q$ corresponding to $\iota$ such that 
$\rho \circ h = id$.  Then by Proposition 7.3 of \cite{tanaka1} we can obtain the connection form $\chi$ defined by 
\[\chi (X) = \omega(h_*X)_{0} \quad (X \in TP_tM). \] 
Thus we obtain the following map 
\[\Phi \!:\! \{\mbox{Grassmannian Cartan connections of type $(n, m)$ on $M$} \}/_\sim \to   
\{ (P_tM, [\chi])\},\]  
where $[\chi]$ denotes a certain equivalence class of $\chi$ defined in \cite[p.128]{tanaka1}. 
Especially when $m \neq 1$ and $n \neq 1$, there is a one-to-one correspondence between 
the set of the isomorphism classes of normal Grassmannian Cartan connections over $M$ 
and the set of $\GL(n) \otimes \GL(m)$-structures on $M$ (see section 9 and Theorem 10.2 of \cite{tanaka1}).  

  
Now we consider the case $m = 1$.  
A Grassmannian Cartan connection $(Q, \omega)$ of type $(n, 1)$ is called a projective Cartan connection. 
Especially a normal projective Cartan connection induces   
a projective equivalence class 
of torsion-free linear connections $(\mathcal{L}(M), [\chi])$, which we call a projective structure.  
Now assume $n > 1$.  By the restriction of $\Phi$ to the normal case with $m = 1$   
gives the following one-to-one correspondence:  
\begin{eqnarray*}
\Phi_{m=1}: \hspace{-0.5cm} &&\{\mbox{normal projective Cartan connections on $M$}\}/_\sim \\ 
&&\to \{\mbox{projective structures on $M$}\}. 
\end{eqnarray*} 
For more details of projective structures we refer the reader to \cite{tanaka3}, \cite{tanaka1}, \cite{nomizu-sasaki} and 
\cite{Agaoka}. 
A projective structure $[\chi]$ on $M$ is said to be projectively flat if $\chi$ is locally projectively equivalent to a flat affine connection.
The map $\Phi_{m=1}$ is restricted to the bijective between the set of flat projective Cartan connections and the set of projective structures which are projectively flat. 
(cf. Theorem 9.2 in \cite{tanaka1} and Proposition 1.5.2 in \cite{cap-slovak}).  
The existence of a flat Grassmannian structure of type $(n, m)$ ($n, m \geq  2$) on $M$ should be also described by the terminology of $\GL(n) \otimes \GL(m)$-structure.  However the author does not know it.  

\medskip
\noindent
{\bf 2.2 \ Subgeometry}

\medskip 
\noindent 
For the later argument, we introduce the notion of subgeometry, following \cite{goldman}.  
Let $A/B$ and $A'/B'$ be real homogeneous spaces. We say that $A/B$ is a subgeometry of $A'/B'$ if there exists a Lie group homomorphism $F:A \to A'$ satisfying the following conditions: 

(1) $F(B) \subset B'$, 

(2) the induced map $\hat{F}: A/B \to A'/B'$ is a local diffeomorphism.  

Let us denote the Lie algebra of $A$ by $\da$, the one of $B$ by $\db$. 
Likewise we define $\da'$ and $\db'$ for $A'$ and $B'$ respectively. 
Let $\Lambda$ and $\Lambda'$ be Maurer-Cartan forms of $A$ and $A'$. 
Then $F$ gives a bundle homomorphism corresponding to $F|_B: B \to B'$ and satisfies 
$F^* \Lambda' = dF \circ \Lambda$. 

\smallskip
\begin{prop}\label{inducing cartan connection} 
Let $(Q, \omega)$ be a Cartan connection of type $A/B$ on $M$. 
Then there exists a Cartan connection $(Q', \omega')$ of type $A'/B'$ on $M$.  
\end{prop}
\vspace{-2mm}
\begin{proof} 
The proof of this Proposition is same as the one of Theorem 1.5.15 of \cite{cap-slovak}. 
Thus we only explain the construction of a Cartan connection $(Q', \omega')$ of type $A'/B'$ on $M$. 
Since $B$ acts on $B'$ via $F$,  
from the given principal bundle $Q$ we obtain the extended bundle $Q' = Q \times_B B'$. 
The bundle homomorphism $\widetilde{F}: Q \to Q'$ 
is defined by $u \mapsto [u,e]$, which corresponds to 
the restriction of $F$ to $B$.   
Next we define a $\da'$-valued 1-form $\omega'$ on $\widetilde{F}(Q)$ by 
\[\omega'_{[u,e]}(\widetilde{F}_* X + Z^*) = dF \circ \omega(X) + Z \ \ (X \in T_uQ, Z \in \db').\]   
We enlarge this definition to the whole of $Q'$ by 
\[\omega'_{[u,c]} = R_{c^{-1}}^* Ad(c^{-1}) \omega_{[u,e]}  \ \ (c \in B').\]
This definition is well defined and  
we can verify $(Q',\omega')$ gives a Cartan connection of type $A'/B'$ on $M$.    
\end{proof} 
\begin{df}
Let $(Q, \omega)$ and $(Q', \omega')$ be  Cartan connections  
of type $A/B$ and $A'/B'$ respectively on $M$. Then we call $(Q, \omega)$ a subgeometry of $(Q', \omega')$ if there exists a bundle homomorphism $\iota: Q \to Q'$ corresponding to $F|_B: B \to B'$ 
such that $\iota$ induces the identity map between the base spaces and $\iota^* \omega' = dF \circ \omega$. 
\end{df}

\smallskip
In Proposition \ref{inducing cartan connection} a given Cartan connection $(Q, \omega)$ of type $A/B$ induces 
$(Q', \omega')$ of type $A'/B'$, and  $(Q, \omega)$ is a subgeometry of $(Q', \omega')$. 

\medskip 
\begin{prop}\label{flatness of subgeometry} 
Assume that a Cartan connection $(Q, \omega)$ is a subgeometry of $(Q', \omega')$. 
If $(Q, \omega)$ is flat, then $(Q', \omega')$ is also flat. Moreover when the differential 
$dF: \da \to \db$ is an injective homomorphism, the converse is also true. 
\end{prop}
\begin{proof}
We compute the curvature form $\Omega'$ of $(Q',\omega')$.  
Pulling back $\Omega'$ by $\iota$ yields
\begin{eqnarray*}
\iota^*\Omega' &=& \iota^*(d\omega' + \frac{1}{2}[\omega', \omega']) \\
&=& dF(d\omega + \frac{1}{2}[\omega, \omega]) \\
&=& dF(\Omega). 
\end{eqnarray*}
Hence the assertion of the proposition follows.  
\end{proof}

We fix the complementary subspace $\m$ of $\db$ and $\m'$ of $\db'$. 
Let $\rho$ be the linear isotropy representation of $B$ on the tangent space to $A/B$ at the origin $o$.  
By identifying $T_oA/B$ with $\m$, $\rho$ is 
given by 
$\rho(b)X = Ad(b)X + \db$ for $b \in B$ and $X \in \m$. Thus we obtain the two linear isotropy representations 
$\rho: B \to GL(\m)$ and $\rho': B' \to GL(\m')$.    
We denote the kernel of $\rho$ by $C$ and the one of $\rho'$ by $C'$. Since we assume that 
$A/B$ is a subgeometry of $A'/B'$, there is a homomorphism $F: A \to A'$ whose differential $dF$ induces the  linear isomorphism $\widehat{dF}: \m \to \m'$. 

\medskip
\begin{lem}\label{diagram of isotropy rep}    
There exists an injective homomorphism $\overline{F}: \rho(B) \to \rho'(B')$ defined by 
$\overline{F}: \rho(b) \mapsto \widehat{dF} \circ \rho(b) \circ \widehat{dF}^{-1}$, 
and we have the commutative diagram: 
\[
\xymatrix{
	 B \ar[r]_F \ar[d]_\rho \ar@{}[dr]|\circlearrowleft &  B' \ar[d]_{\rho'} \\
     \rho(B) \ar@{^{(}->}[r]_{\overline{F}} & \rho'(B') 	
     }
\]     
Moreover $F$ is regarded as a bundle homomorphism corresponding to $F: C \to C'$.  
\end{lem}
\begin{proof} 
Firstly we verify $F(C) \subset C'$. 
Assume that $b \in C$. Then $\rho(b)(X+\db) = X + \db$ for $X \in \m$.  
Then $\rho'(F(b))(\widehat{dF}(X)) = Ad(F(b))(dF(X)+\db') = dF(Ad(b)X) + \db' =  dF(X) + \db' = \widehat{dF}(X)$.  
Thus $\rho'(F(b)) = id_{\m'}$, and $F(b) \in C'$.
We define $\overline{F}: \rho(B) \to \rho'(B')$ by $\bar{F}: \rho(b) \mapsto \rho'(F(b))$ for $b \in B$. 
Since $F(C) \subset C'$, this is well defined, moreover we have $\overline{F}(\rho(b)) = \widehat{dF} \circ \rho(b) 
\circ \widehat{dF}^{-1}$. It follows that $\overline{F}$ is injective.
\end{proof}

Let $(Q, \omega)$ and $(Q', \omega')$ be Cartan connections of type $A/B$ and $A'/B'$ respectively. 
Assume that $(Q, \omega)$ is a subgeometry of $(Q', \omega')$. 
We denote the quotient manifold $Q/C$ by $\widetilde{Q}$ and $Q'/{C'}$ by $\widetilde{Q}'$.  
Then we obtain $\rho(B)$-structure $(\widetilde{Q}, \theta)$ and $\rho'(B')$-structure $(\widetilde{Q}', \theta')$ 
(see \cite[p.136]{tanaka2}). 
The projection $\rho: Q \to \widetilde{Q}$ is corresponding to $\rho: B \to \rho(B)$. 
Recall that $(\widetilde{Q}, \theta)$ gives a $\rho(B)$-structure on $M$ as follows: concerning each point  
$\rho(u) \in \widetilde{Q}$, 
$\rho(u)^{-1}$ is regarded as a linear isomorphism $T_{\pi_{Q/C}(\rho(u))}M \to \m$ by  
$\rho(u)^{-1}: {\pi_{Q/C}}_* \rho _* X \mapsto \theta(\rho _* X) = \omega_\m(X)$.   
Hence we obtain the bundle homomorphism  $\widetilde{Q} \hookrightarrow L(M)$ corresponding to the inclusion 
$\rho(B) \to GL(\m)$.  
Likewise we obtain the map $\widetilde{Q}' \hookrightarrow L(M)$, where $L(M)$ is regarded as the set of all linear 
isomorphisms $y:\m' \to T_pM (p \in M)$. 
 
The bundle homomorphism $t:L(M) \to L(M)$ is defined by 
$t: x \mapsto x \circ \widehat{dF}^{-1}$, 
which is corresponding to $GL(\m) \ni A \mapsto \widehat{dF} \circ A \circ \widehat{dF}^{-1} \in  GL(\m')$. 

\medskip
\begin{prop}\label{red of G-structure}
The $\rho(B)$-structure $\widetilde{Q}$ 
is a reduction of $\rho'(B')$-structure $\widetilde{Q}'$ 
i.e. $1)$ there exists a bundle homomorphism $\bar{\iota}: \widetilde{Q} \hookrightarrow  \widetilde{Q'}$ corresponding to $\bar{F}: \rho(B) \to \rho'(B')$, and $2)$  
$\bar{\iota}^*\theta' = \widehat{dF} \circ \theta$. 
The injection $\bar{\iota}$ is given by the restriction of $t:L(M) \to L(M)$.  
\end{prop} 
\vspace{-2mm}
\begin{proof} 
From assumption 
$(Q, \omega)$ is a subgeometry of $(Q', \omega')$, thus we have a bundle homomorphism $\iota: Q \to Q'$ corresponding to $F|_B: B \to B'$.  
We define a map $\bar{\iota}: \widetilde{Q} \to \widetilde{Q}'$ by $\bar{\iota}: \rho(u) \to \rho' \circ \iota(u)$. Since $F|_B$ is a bundle homomorphism corresponding to $F: C \to C'$,  
$\bar{\iota}$ is well-defined and gives a bundle homomorphism corresponding to $\bar{F}: \rho(B) \to \rho'(B')$.  
Hence we obtain the commutative diagram: 
\begin{equation}\label{commutative diagram 2}
\xymatrix{
	 Q \ar[r]_i \ar[d]_\rho \ar@{}[dr]|\circlearrowleft &  Q' \ar[d]^{\rho'} \\
     \widetilde{Q} \ar@{^{(}->}[r]_{\bar{\iota}} \ar[d]_{\pi_{\widetilde{Q}}} \ar@{}[dr]|\circlearrowleft & \widetilde{Q}' \ar[d]^{\pi_{\widetilde{Q}'}} \\	
     M    \ar[r]_{id}        &  M. 
     }
\end{equation} 

Since $\bar{\iota}$ induces the identity of base spaces and $\bar{F}$ is injective, $\bar{\iota}$ is injective. 
Now we show that $\bar{\iota}^*\theta' = \widehat{dF} \circ \theta$. 
Since $\bar{\iota} \circ \rho = \rho' \circ \iota$,  
pulling back $\bar{\iota}^*\theta'$ by $\rho: Q \to \widetilde{Q}$ yields $\rho^*(\bar{\iota}^*\theta') 
= (\rho' \circ \iota)^* \theta' = \iota*\omega'_{\m'} = \widehat{dF} \circ \omega_{\m}$. 
Hence $\bar{\iota}^*\theta' = \widehat{dF} \circ \theta$. 

By using the inclusion $\widetilde{Q} \hookrightarrow L(M)$ and $\widetilde{Q}' \hookrightarrow L(M)$,  
$t$ and $\bar{\iota}$,  we obtain the following diagram, which will be shown commutative as follows.  
\begin{equation}\label{commutative diagram 3}
\xymatrix{
	 L(M) \ar[r]_t \ar@{}[dr]|\circlearrowright &  L(M) \\ 
     \widetilde{Q} \ar@{^{(}->}[r]_{\bar{\iota}} \ar[u] & \widetilde{Q}' \ar[u] 	 
     }
\end{equation} 
From the equality $\bar{\iota}^*\theta' = \widehat{dF} \circ \theta$, for $\rho(u) \in \widetilde{Q}$ we have 
$\theta'_{\bar{\iota}(\rho(u))}(\bar{\iota}_*\rho _*X) = \widehat{dF} \circ (\rho_*X) = 
\widehat{dF} \circ \omega_{\m}(X)$. Thus $\bar{\iota}(\rho(u))$ gives a linear isomorphism 
$T_{\pi_{\widetilde{Q}'}(\bar{\iota}(\rho(u))}M \to \m'$ by $\bar{\iota}(\rho(u)): 
{\pi_{\widetilde{Q'}}}_* \bar{\iota}_* \rho _* X \mapsto \widehat{dF} \circ \omega_{\m}(X)$. 
Since the diagram 
\eqref{commutative diagram 2} 
is commutative, we have ${\pi_{\widetilde{Q}'}}_* \bar{\iota}_* \rho _* X = 
{\pi_{\widetilde{Q}}}_* \rho_* X$.  
Therefore $\bar{\iota}(\rho(u)) = \rho(u) \circ \widehat{dF}^{-1}$.  
On the other hand $t \circ \rho(u) =  \rho(u) \circ \widehat{dF}^{-1}$, and hence the diagram  
\eqref{commutative diagram 3} is commutative. 
\end{proof} 
Let $X'$ be a homogeneous space of $A'$. We choose $B'$ as 
the isotropy subgroup at a point $v$ in $X'$. 
If we are given a subgroup $A \subset A'$, we can consider the isotropy subgroup $A_v$ of $A$ at $v$. 
Then $A/A_v$ gives a subgeometry of $A'/B'$.  
Henceforth we say that $A/B$ is a subgeometry of $A'/B'$ if this condition is satisfied: 
$A$ is a subgroup of $A'$ and $B$ is the isotropy subgroup at $v$. 

\section{Castling transformations} 
In this section we establish a castling transformation of projective structures. 
Let $G$ be a Lie subgroup of $\PL(l)$. We consider the homomorphism $F: G \times \PL(m) \hookrightarrow 
\PL(\R^l \otimes \R^m)$ defined by $(g, A) \mapsto g \otimes A$. 
By $F$ we regard $G \times \PL(m)$ as 
a subgroup of $\PL(\R^l \otimes \R^m)$. When we identify $\R^l \otimes \R^m$ with  
$\underbrace{\R^l \oplus \cdots \oplus \R^l}_m$, 
$G \times \PL(m)$ acts on $P(\R^l \otimes \R^m)$    
by $(g, A).v := gv{}^tA$  for $v = (v_1, \ldots, v_m) \in P(\R^l \otimes \R^m)$. 
Assume that $n = l-m \geq  0$. Denote by $V_{m,l}$ a projective Stiefel manifold, which consists of 
projective frames of $m$-dimensional subspaces of $\R^l$. 

\medskip
\begin{prop}\label{cast of rep} 
Let $v$ be a point in $P(\R^l \otimes \R^m)$. 
Then the rank of $v = (v_1, \ldots, v_m)$ is $m$ and 
the homogeneous space $G/G_{<v>}$ is a subgeometry of $Gr_{m, l}=\PL(l)/\PL(l)_{<v>}$ if 
and only if $G \times \PL(m)/G \times \PL(m)_{v}$ is a subgeometry of the projective space $P(\R^l \otimes \R^m) = 
\PL(\R^l \otimes \R^m)/\PL(\R^l \otimes \R^m)_v$ defined by $F: G \times \PL(m) \hookrightarrow 
\PL(\R^l \otimes \R^m)$.   
\end{prop} 
\begin{proof} 
The proof follows the idea of Proposition 6 of section 2 in \cite{sato-kimura}. 
We can prove this proposition by showing that the following four assertions are equivalent.  
\begin{enumerate}[(1)]
\item 
$G/G_{<v>}$ is a subgeometry of the Grassmannian manifold $\PL(l)/\PL(l)_{<v>}$.  

\item
$G.<\!v\!>$ gives an open orbit in $G_{m,l}$. 

\item
$G \times \PL(m).v$ gives an open orbit in $V_{m,l}$.  

\item
$G \times \PL(m)/{G \times \PL(m)_{v}}$ is a subgeometry of the projective space $\PL(\R^l \otimes \R^m)/\PL(\R^l \otimes \R^m)_v$. 
\end{enumerate} 
The proof of $(1) \Leftrightarrow (2)$ is easy. 
To prove $(2) \Leftrightarrow (3)$, we consider the fiber bundle 
\[
	\xymatrix{
      V_{m,l} \ar[d]^{\pi} & PL(m) \ar[l] \\
      G_{m,l}
} 
\]
The natural projection $\pi$ is continuous and open map. It follows $(2) \Leftrightarrow (3)$.  
To prove $(3) \Leftrightarrow  (4)$, we observe that the manifold $V_{m,l}$ is naturally imbedded 
into $P(\R^l \otimes \R^m)$ with respect to the relative topology, indeed $V_{m,l}$ is an open submanifold in 
$P(\R^l \otimes \R^m)$.  Next we consider the assertion $(3)'$: $G \times \PL(m).v$ gives an open orbit in 
$P(\R^l \otimes \R^m)$. If $(3)'$ holds true,  then $v$ must belong to $V_{m,l}$. 
Hence we have $(3) \Leftrightarrow (3)'$. The proof of the equivalence $(3)' \Leftrightarrow  (4)$ is same 
as the one of $(1) \Leftrightarrow (2)$.  
\end{proof}

Now assume that $l-m \geq 1$. 
We fix a point $v$ of $V_{m,l}$ and one $v^\perp$ of $V_{l-m,l}$ such that   
the subspace $<\!v^\perp\!>$ spanned by $v^\perp$ 
is orthogonal to the one $<\!v\!>$ spanned by $v$. 
Let us define the isomorphism $*:\PL(l) \to \PL(l)$ by $g \mapsto {}^tg^{-1}$. Then $*$ gives 
the isomorphism between 
$\PL(l)_{<v>}$ and $\PL(l)_{<v^\perp>}$. We denote the differential of $*$ by the same symbol, and we have 
$*(A) = -{}^tA$ for $A \in \sll(l)$. 
Then $G.<\!v\!>$ gives an open orbit in $Gr_{m,l}$ if and only if $*G.<\!v^\perp\!>$  gives an open orbit in 
$Gr_{l-m,l}$. Hence we obtain the following trivial fact: 

\medskip
\begin{prop}\label{cast of rep2}
$G/G_{<v>}$ is a subgeometry of $Gr_{m, l}$ iff 
$*G/*G_{<v^\perp>}$ is a subgeometry of $Gr_{l-m, l}$. 
\end{prop}

\smallskip
Let $\rho: H \to \GL(\C^l)$ be a rational representation of a complex linear algebraic group $H$. Then originally 
the transformation 
\[
\rho \otimes id: H \times \GL(m) \to \GL(\C^l \otimes \C^m)  
\Leftrightarrow  \rho^* \otimes id: H \times \GL(l-m) \to \GL(\C^l \otimes \C^{l-m})\]  
is called a castling transformation in \cite{sato-kimura}. It has been proved that 
$\rho$ gives a prehomogeneous vector space iff $\rho^*$ gives a prehomogeneous vector space.    

Now we define the castling transformation of Cartan connections. 
Denote by $\g$ the Lie algebra of $G$. 
Assume $v \in V_{m,l}$ and that $G/G_{<v>}$ is a subgeometry of $Gr_{m, l}$. We denote by $\Lambda_1$ the Maurer-Cartan form of $PL(m)$. 

\medskip
\begin{prop}\label{cast of geo1} 
Denote by $Q$ 
a principal fiber bundle over $M$ with structure group 
$G_{<v>}$ and by $\omega$ a $\g$-valued 1-form on $Q$. 
Then the following are equivalent. 
\begin{enumerate}[1.] 
\item 
$(Q, \omega)$ is a Cartan connection of type $G/G_{<v>}$ on $M$.   

\item 
$(Q \times \PL(m), \omega \times \Lambda_1)$ is a Cartan connection of 
type $G \times PL(m)/{G \times PL(m)}_{v}$ on a manifold $N$.     
\end{enumerate}   
$(Q, \omega)$ is flat iff $(Q \times \PL(m), \omega \times \Lambda_1)$ is flat. 
\end{prop} 
\begin{proof} 
$\mathit{1} \Leftrightarrow \mathit{2}$: 
Assume the assertion 1. Then $Q \times \PL(m)$ is regarded as a principal fiber bundle over $M$ with 
structure group $G_{<v>} \times \PL(m)$. Since $G \times \PL(m)_{v}$ is a closed subgroup of 
$G_{<v>} \times \PL(m)$, we have the quotient $Q \times \PL(m)/{G \times \PL(m)_{v}}$ over which 
$Q \times \PL(m)$ 
is regarded as a principal fiber bundle with structure group $G \times \PL(m)_{v}$. 
Then we can directly check $(Q \times \PL(m), \omega \times \Lambda_1)$ gives a Cartan connection of 
type $G \times \PL(m)/G \times \PL(m)_{v}$ on $Q \times \PL(m)/{G \times \PL(m)_{v}}$.  
Conversely we assume the assertion $\mathit{2}$. 
We can directly check $(Q, \omega)$ gives a Cartan connection of type $G/G_{<v>}$ on $M$.  

Now we prove the equivalence of flatness between $\mathit{1}$ and $\mathit{2}$. We first observe that    
$(d\omega + \frac{1}{2}[\omega, \omega], d\Lambda_1 + \frac{1}{2}[\Lambda_1, \Lambda_1])$ gives  
$\g \times \sll(m)$-valued 2-form on  $Q \times PL(m)$. 
We can directly verify    
\[d(\omega \times \Lambda_1) + \frac{1}{2}[\omega \times \Lambda_1, \omega \times \Lambda_1]
= (d\omega + \frac{1}{2}[\omega, \omega], d\Lambda_1 + \frac{1}{2}[\Lambda_1, \Lambda_1]).\]
Note that since $\Lambda_1$ is the Maurer-Cartan form of  
$PL(m)$, we have $d\Lambda_1 + \frac{1}{2}[\Lambda_1, \Lambda_1] =0$. 
Hence $(Q, \omega)$ is flat if and only if 
$(Q \times \PL(m), \omega \times \Lambda_1)$ is flat. 
\end{proof} 

\begin{prop}\label{cast of geo2} 
Let $(Q, \omega)$ be a Cartan connection of type $G/G_{<v>}$ on $M$.   
Then 
$(Q, *\omega)$ gives a Cartan connection of type $*G/*G_{<v^\perp>}$ on $M$. 
$(Q, \omega)$ is flat iff $(Q, *\omega)$ is flat. 
\end{prop} 
\begin{proof}  
From the assumption $G_{<v>}$ acts on $Q$ on the right.   
We define the action of $*G_{<v^\perp>}$ on $Q$ by $u \cdot g := u*g$ ($u \in Q$, $g \in *G_{<v^\perp>}$).  
Then the bundle $Q$ is regarded also as a principal fiber bundle over $M$ with structure group $*G_{<v^\perp>}$. 
Moreover we define a one-form $*\omega$ by the composite of $*: \sll(l) \to \sll(l)$ and $\omega$. 
On the other hand since $G/G_{<v>}$ is a subgeometry of $*G/*G_{<v^\perp>}$, the Cartan connection $(Q, \omega)$ induces a Cartan connection $(Q_c, \omega_c)$ of type 
$*G/*G_{<v^\perp>}$ by  
Proposition \ref{inducing cartan connection}.  
This induced Cartan connection is isomorphic with $(Q, *\omega)$.
Hence by Proposition \ref{flatness of subgeometry} $(Q, \omega)$ is flat if and only if $(Q, *\omega)$ is flat. 
\end{proof} 

Now we assume the same assumption as Proposition \ref{cast of geo1}: 
Denote by $Q$ a principal fiber bundle over $M$ with structure group 
$G_{<v>}$ and by $\omega$ a $\g$-valued 1-form on $Q$. 
Then combining Propositions \ref{cast of rep}, \ref{cast of rep2}, \ref{cast of geo1} and \ref{cast of geo2} yields the next theorem.  

\medskip
\begin{thm}\label{castling transformation} The following are equivalent. 
\begin{enumerate}[1.]
\item 
$(Q \times \PL(m), \omega \times \Lambda_1)$ is a Cartan connection of type 
$G \times PL(m)/G \times PL(m)_{v}$.  
 
\item 
$(Q, \omega)$ is a Cartan connection of type $G/G_{<v>}$ on $M$. 

\item
$(Q, *\omega)$ is a Cartan connection of type $*G/*G_{<v^\perp>}$ on $M$. 

\item 
$(Q \times \PL(l-m), *\omega \times \Lambda_1)$ is a Cartan connection of type $*G \times PL(l-m)/$ 
$*G \times PL(l-m)_{v^\perp}$. 
\end{enumerate} 
Moreover the flatness of the above Cartan connections are equivalent. 
The Cartan connections 1 and 4 \mbox{\rm(}resp. 2 and 3\mbox{\rm)} are subgeometries of projective \mbox{\rm(}resp. Grassmannian\mbox{\rm)} Cartan connections.   
\end{thm} 

\smallskip
In this theorem we omit the base space $N$ of the Cartan connection $(Q \times \PL(m), \omega \times \Lambda_1)$ since $N$  
is diffeomorphic to the quotient $Q \times \PL(m)/G \times PL(m)_{v}$. This quotient is described as follows by using 
a Grassmannian structure on $M$. 

\medskip
\begin{prop} 
Suppose that $(Q, \omega)$ gives a Cartan connection of type $G/G_{<v>}$ on $M$. 
Then the base space $N$ of $(Q \times \PL(m), \omega \times \Lambda_1)$ is a principal fiber bundle over $M$ with group $PL(m)$.   
Moreover $(Q, \omega)$ induces a $\GL(n) \otimes \GL(m)$-structure $P_tM$ on $M$, and $N$ is diffeomorphic to the quotient of $P_tM$ by $\GL(n) \otimes \GL(1)$. 
\end{prop}
\vspace{-2mm} 
\begin{proof} 
We fix the complementary subspace $\m$ of $\g_{<v>}$ in $\g$, then the natural inclusion 
$\iota: G \to \PL(l)$ gives the linear isomorphism $\hat{d\iota}: \m \to M(n, m)$.  
We denote the isotropy representation of $G/G_{<v>}$ by  
$\rho: G_{<v>} \hookrightarrow \GL(\m)$. Then from the assumption 
$(Q/{ker\rho}, \theta)$ gives a $\rho(G_{<v>})$-structure 
$\widetilde{Q} \subset \mathcal{L}(M)$. 
From Proposition \ref{inducing cartan connection} we obtain the induced Grassmannian Cartan connection  
$(Q', \omega')$, which induces a $\GL(n) \otimes \GL(m)$-structure $P_tM$ on $M$.  
Now by Proposition \ref{red of G-structure} the natural inclusion 
$\iota: Q \to Q'$ induces the injective 
$\bar{\iota}: \widetilde{Q} \to P_tM$.  
Consequently we obtain the following diagram. 
\begin{equation*} \label{commutative diagram 4}
\xymatrix{
	 Q \ar@{^{(}->}_i[r] \ar[d]_\rho \ar@{}[dr]|\circlearrowleft &  Q' \ar[d]_{\tilde{\rho}} \\ 
     \widetilde{Q} \ar@{^{(}->}[r]_{\bar{\iota}} & 
     P_tM. 
     } 
\end{equation*} 
Thus we obtain the map $\bar{\iota} \circ \rho: Q \to P_tM$ corresponding to the restriction $\tilde{\rho}|_{G_{<v>}}$. For a matrix $A \in \GL(m)$ denote by $\bar{A}$ the image of the homomorphism 
$\GL(m) \to \PL(m)$. 
We define the map $\Phi: Q \times \PL(m)/{{G \times \PL(m)}_{v}} \to P_tM/{\GL(n) \otimes \GL(1)}$ by 
\[\Phi: (u,\bar{A}) {G \times \PL(m)}_v  \mapsto \bar{\iota} \circ \rho(u) I_n \otimes A^{-1} \GL(n) \otimes \GL(1).\]
This definition is well defined. Moreover the map $\Phi$ is a diffeomorphism.  
Now we observe that $\PL(m)$ acts 
on $Q \times \PL(m)/{{G \times \PL(m)}_{v}}$ and  
$P_tM/{\GL(n) \otimes \GL(1)}$ on the right as follows: 
$(u, \bar{A}){G \times \PL(m)}_v \cdot \bar{B} := (u, \overline{B^{-1} A}){G \times \PL(m)}_v$ and 
$x \ \GL(n) \otimes \GL(1) \cdot \bar{B} := x \ I_n \otimes B \ \GL(n) \otimes \GL(1)$. 
We can check that by these actions both $Q \times \PL(m)/{{G \times \PL(m)}_{v}}$ and  
$P_tM/{\GL(n) \otimes \GL(1)}$ can be regarded as principal fiber bundles over $M$ with 
structure group $PL(m)$. Then $\Phi$ gives a bundle isomorphism. 
If we have $m =1$, the base space of $(Q \times \PL(m), \omega \times \Lambda_1)$ is same as $M$, and 
the base space of $(Q \times \PL(n), *\omega \times \Lambda_1)$ is isomorphic to $\mathcal{L}(M)/GL(1)$. 
\end{proof}

\newpage
\begin{cor}\label{cast of geo} 
We call the transformation 
\[
\mathrm{(I)} \ (Q \times PL(m), \omega \times \Lambda_1) \leftrightarrow \mathrm{(I\hspace{-.1em}I)} \ (Q \times PL(l-m), *\omega \times \Lambda_1) 
\] 
a castling transformation of projective structures.  
\mbox{\rm(I)} is a subgeometry of a \mbox{\rm(}resp. flat\mbox{\rm)} projective Cartan connection iff   
\mbox{\rm(I\hspace{-.1em}I)} is a subgeometry of a \mbox{\rm(}resp. flat\mbox{\rm)} projective Cartan connection. 
\end{cor} 

\smallskip 
When $m =1$, the Cartan connection $(Q, \omega)$ itself gives a projective Cartan connection, whose model space is $PL(l)/PL(l)_{v}$ and $v$ belongs to $V_{1,l}$. 
By the castling transformation of $(Q, \omega)$, we obtain the Cartan connection $(Q \times PL(l-1), *\omega \times \Lambda_1)$, which is a subgeometry of a projective 
Cartan connection. 

Theorem \ref{castling transformation} is described by the following commutative diagram. 
We assume $(Q, \omega)$ is a Grassmannian Cartan connection on $M$. 
\[
\xymatrix{ 
& (Q \times PL(m), \omega \times \Lambda_1) \ar[dd]^{\pi_m} \\
(Q, \omega) \ar[ur] \ar[rd] \ar[d]_\rho & \\ 
P_tM \ar[d]_\pi \ar[r]_{r_m} & {M}_{m} \ar[ld]^{\bar{\pi}} \\
M & \\ 
(Q, *\omega) \ar[u] &  
}
\] 
The manifold ${M}_m$ denotes a quotient manifold $P_tM/\GL(n) \otimes \GL(1)$, and 
we denote by $r_m$ the projection $P_tM \to P_tM/\GL(n) \otimes \GL(1)$. 
The manifold ${M}_m$ is naturally isomorphic 
with $Q \times_{\PL(l)_{<v>}} \PL(m)$. Thus we obtain a natural projection from $Q$ to ${M}_m$ and 
there exists a natural inclusion $Q \to Q \times \PL(m)$. 
Denote by $\pi_m$ a projection from $Q \times \PL(m)$ to the quotient manifold  
$Q \times \PL(m)/\PL(l) \times \PL(m)_{v}$. From the proof of Proposition \ref{cast of geo} the quotient space  
$Q \times \PL(m)/\PL(l) \times \PL(m)_{v}$ is identified with ${M}_m$ by the bundle isomorphism given by 
$\pi_m(z, g) \leftrightarrow r_m \circ \rho(z)g^{-1}$.

\section{Successive castling transformations}
In this section we give two fundamental procedures to do castling transformations successively.  
The product group  $PL(l) \times \prod_{i=1}^j PL(k_i)$ naturally acts on 
$P(\R^l \otimes \bigotimes_{i=1}^j \R^{k_i})$ by the 
tensor product, namely via the inclusion 
$\imath: PL(l) \times \prod_{i=1}^j PL(k_i) \hookrightarrow PL(\R^l \otimes \bigotimes_{i=1}^j \R^{k_i})$ given by 
$\imath (\bar{g}, \bar{A_1}, \ldots, \bar{A_j}) = g \otimes A_1 \otimes \cdots \otimes A_j$. 
We denote the natural basis of $\R^{l}, \R^{k_1}, \ldots, \R^{k_j}$ by $\{e_{i_0}\}_{i_0=1}^l, 
\{e_{i_1}\}_{i_1= 1}^{k_1}, \ldots$, $\{e_{i_j}\}_{i_j= 1}^{k_j}$.  
Then a point $w$ in $\R^l \otimes \bigotimes_{i=1}^j \R^{k_i}$ 
is written by $w = \sum_{i_0, i_1,\ldots, i_j}C_{i_0 i_1 \cdots i_j} 
e_{i_0}$ $\otimes e_{i_1} \otimes \cdots \otimes e_{i_j}$, where $C_{i_0 i_1 \cdots i_j}$ is a coefficient. 
Now let $\sigma$ be an element of symmetric group of $\{1,2, \cdots j\}$. Then $\sigma$ induces a natural linear 
isomorphism $\R^l \otimes \bigotimes_{i=1}^j \R^{k_i} \to \R^l \otimes \bigotimes_{i=1}^j \R^{\sigma(k_i)}$, which is 
defined by $\sigma(w) = 
\sum_{i_0, i_1,\ldots, i_j}C_{i_0 i_1 \cdots i_j} e_{i_0} \otimes e_{i_\sigma(1)} \otimes \cdots \otimes 
e_{i_\sigma(j)}$.  The map $\sigma$ induces a diffeomorphism 
$P(\R^l \otimes \bigotimes_{i=1}^j \R^{k_i}) \to 
P(\R^l \otimes \bigotimes_{i=1}^j \R^{\sigma(k_i)})$.  
The group $PL(l) \times \prod_{i=1}^j PL(k_{\sigma(i)})$ naturally acts on 
$P(\R^l \otimes \bigotimes_{i=1}^j \R^{\sigma(k_i)})$, 
and the action satisfies the condition 
$(g, A_{\sigma(1)}, \ldots, A_{\sigma(j)}) \cdot \sigma(w)  = \sigma((g, A_1, \ldots, A_j) \cdot w)$.  
It is easy to prove the following.   

\smallskip
\begin{prop}\label{successive castling transf of rep} 
Assume that a point $w$ in $P(\R^l \otimes \bigotimes_{i=1}^j \R^{k_i})$ gives an open orbit of 
$PL(l) \times \prod_{i=1}^j PL(k_i)$. Then we have the following: 
\begin{enumerate}[1.] 
\item  
For any permutation $\sigma$ of $\{1,2, \cdots j\}$ the product group     
$PL(l) \times \prod_{i=1}^j PL(k_{\sigma(i)})$ admits an open orbit given by $\sigma(w)$. 
Isotropy subgroups $PL(l) \times \prod_{i=1}^jPL(k_i)_{w}$ and 
$PL(l) \times \prod_{i=1}^j PL(k_{\sigma(i)})_{\sigma(w)}$ 
are isomorphic. 
\item 
The product group $PL(l) \times \prod_{i=1}^j PL(k_i) \times PL(1)$ acts on $P(\R^l \otimes \bigotimes_{i=1}^j \R^{k_i} \otimes \R)$,    
which is identified with 
$P(\R^l \otimes \bigotimes_{i=1}^j \R^{k_i})$ 
naturally. Via this identification $w$ gives an open orbit of 
$PL(l) \times \prod_{i=1}^j PL(k_i) \times PL(1)$, and isotropy subgroups 
$PL(l) \times \prod_{i=1}^j PL(k_i)_{w}$ and 
$PL(l) \times \prod_{i=1}^j PL(k_i) \times PL(1)_{w}$ are isomorphic. 
\end{enumerate}
\end{prop} 
From the Lie group $PL(l) \times \prod_{i=1}^j PL(k_i)$ we can obtain several new Lie groups by 
Propositions \ref{successive castling transf of rep} and castling transformations. 
For example let $L$ be a Lie subgroup of $PL(3)$ such that $L$ admits an open orbit $L.x$ in $P(\R^3)$. Then 
we obtain the sequence of new groups $*L \times PL(2)$, $L \times PL(2) \times PL(5)$, 
$*L \times PL(5) \times PL(13)$, $L \times PL(5) \times PL(13) \times PL(194)$, which admit open orbits in 
projective spaces.    
Note that $*L \times PL(2)$ is regarded as a subgroup $*L \otimes PL(2)$ of $PL(6)$.  
Next we apply Proposition \ref{successive castling transf of rep} to Cartan connections. 

\smallskip
\begin{prop}\label{successive castling transf of Cartan} 
Let $Q$ be a manifold equipped with a $\g$-valued 1-from $\omega$.   
Assume that $(Q \times \prod_{i=1}^j PL(k_i), \omega \times \prod_{i=1}^j \Lambda_1)$ gives a Cartan connection over a manifold $N$ of type 
\[PL(l)  \times \prod_{i=1}^j PL(k_i)/{PL(l) \times \prod_{i=1}^j PL(k_i)_w},\] 
which is a 
subgeometry of $PL(\R^l \otimes \bigotimes_{i=1}^j \R^{k_i})/PL(\R^l \otimes \bigotimes_{i=1}^j \R^{k_i})_w$  given by 
$\imath:PL(l) \times \prod_{i=1}^j PL(k_i)$ $\hookrightarrow PL(\R^l \otimes \bigotimes_{i=1}^j\R^{k_i})$.    
Then we have the following: 
\begin{enumerate}[1.] 
\item 
For any permutation $\sigma$ of $\{1,2, \cdots j\}$, the pair 
$(Q \times \prod_{i=1}^j PL(k_{\sigma(i)}), \omega \times \prod_{i=1}^j \Lambda_1)$ gives a Cartan connection over $N$ of type 
\[PL(l) \times \prod_{i=1}^j PL(k_{\sigma(i)})/
{PL(l) \times  \prod_{i=1}^j PL(k_{\sigma(i)})_{\sigma(w)}},\] 
which is a 
subgeometry of $PL(\R^l \otimes \bigotimes_{i=1}^j \R^{k_{\sigma(i)}})/ 
PL(\R^l \otimes \bigotimes_{i=1}^j \R^{k_{\sigma(i)}})_{\sigma(w)}$.  

\item 
$(Q \times \prod_{i=1}^j PL(k_i) \times PL(1), \omega \times \prod_{i=1}^j \Lambda_1 \times \Lambda_1)$ gives a Cartan connection over a manifold $N$ of type 
\[PL(l) \times \prod_{i=1}^j PL(k_i) \times PL(1)/
{PL(l)  \times \prod_{i=1}^j PL(k_i) \times PL(1)_w},\] 
which is a subgeometry of $PL(\R^l \otimes \bigotimes_{i=1}^j \R^{k_i} \otimes \R)/PL(\R^l \otimes 
\bigotimes_{i=1}^j \R^{k_i} \otimes \R)_w$. 
\end{enumerate} 
\end{prop}

We consider the castling transformation of 
$(Q \times \prod_{i=1}^j PL(k_i), \omega \times \prod_{i=1}^j \Lambda_1)$. 
The group $PL(l) \times \prod_{i=1}^{j} PL(k_i)$ can be identified with a subgroup of 
$PL(\R^l \otimes \bigotimes_{i=1}^j \R^{k_i})$ 
by the map $F: (g, A_1, \ldots, A_{j}) \mapsto g \otimes A_1 \otimes \cdots \otimes A_{j}$. 
Then one form $\omega \times \prod_{i=1}^j \Lambda_1$ can be identified with the 1-form 
$dF \circ \omega \times \prod_{i=1}^j \Lambda_1$, which is computed as follows:  
\vspace{-3mm}
\begin{eqnarray*} 
dF \circ \omega \times \prod_{i=1}^j \Lambda_1(X, Y_1, \cdots, Y_{j}) &=& 
\omega(X) \otimes I \otimes \cdots \otimes I + I \otimes Y_1 \otimes \cdots \otimes I  \\
&& {} + \cdots + I \otimes \cdots \otimes I \otimes Y_{j}.  
\end{eqnarray*} 
Thus the group isomorphism $*$ of $PL(\R^l \otimes \bigotimes_{i=1}^{j} \R^{k_i})$ is restricted to the subgroup 
$PL(l) \times \prod_{i=1}^{j} PL(k_i)$ by 
\vspace{-2mm}
\[* : (g, A_1, \ldots, A_{j}) \mapsto (*g, *A_1, \ldots, *A_{j}). \] 
The Lie algebra isomorphism $*$ of $\sll(\R^l \otimes \bigotimes_{i=1}^{j} \R^{k_i})$ is restricted to 
$\sll(l) \times \prod_{i=1}^{j} \sll(k_i)$ and we have 
\vspace{-5mm}
\[* (\omega \times \prod_{i=1}^j \Lambda_1) = * \omega \times \prod_{i=1}^j * \Lambda_1. \]
By using Propositions \ref{cast of geo} and \ref{successive castling transf of Cartan}, we can apply castling 
transformations successively. For example let us consider a projective 
Cartan connection $(Q, \omega)$ over 2-dimensional manifold $M$. The bundle $Q$ is a principal fiber bundle over  
$M$ with structure group $PL(3)_v$, where $v$ is an element of $P(\R^3)$   
and $\omega$ is a $\sll(3)$-valued 1-form. 
Then for example we obtain the following sequence by successive 
castling transformations: 

\begin{eqnarray*} 
&& (1) \ (Q, \omega) \longrightarrow (2) \  (Q, *\omega) \longrightarrow (3) \ (Q \times PL(2), *\omega \times \Lambda_1) \\ 
&& \longrightarrow  (4) \ (Q \times PL(2), \omega \times *\Lambda_1) \longrightarrow  
(5) \ (Q \times PL(2) \times PL(5), \omega \times *\Lambda_1 \times \Lambda_1).  
\end{eqnarray*}
In this process 
we fix a point $v^\perp \in V_{2, 3}$ and identify $v^\perp$ with a point $w$ of $P(\R^3 \otimes \R^2)$, and fix a point $w^\perp \in V_{5,  3 \cdot 2}$. 
Now consider the Grassmannian manifold $PL(l)/PL(l)_{<v>}$, and we denote by 
$\rho$ the isotropy representation of $PL(l)_{<v>}$. Here $PL(l)_{<v>}$ is expressed with respect to a basis obtained from $(v, v^\perp)$. 
Then $\rho$ takes values in 
$GL(n) \otimes GL(m)$. When $\rho(g)$ is expressed as  $A \otimes B$, we define a projective linear representation 
$\rho_m: PL(l)_{<v>} \to PL(m)$ by  $\rho_m(g) = {B}$. 
The isotropy group $PL(l) \times PL(m)_{v}$ is equal to the set $\{(g, \rho_m(g)) \mid g \in PL(l)_{<v>}\}$. 
Thus the two isotropy groups $PL(l)_{<v>}$ and $PL(l) \times PL(m)_{v}$ are isomorphic.

Here we omitted the process of 1 and 2 in Proposition \ref{successive castling transf of Cartan}.  In detail we omitted the 
Cartan connections $(Q \times PL(1), \omega \times \Lambda_1)$ and 
$(Q \times PL(2) \times PL(1), *\omega \times \Lambda_1 \times \Lambda_1)$.   
The structure group of each Cartan connection is given by 
\begin{eqnarray*}
&(1)& PL(3)_v,  \\
&(2)& PL(3)_{<v^\perp>} = *PL(3)_v, \\ 
&(3)& PL(3) \times PL(2)_{v^\perp} = \{(*g, \rho_2(*g)) \mid g \in PL(3)_v\}, \\
&(4)& PL(3) \times PL(2)_{<w^\perp>} = \{(g, *\rho_2(*g))\}, \\
&(5)& PL(3) \times PL(2) \times PL(5)_{w^\perp} = \{(g, *\rho_2(*g), \rho_5(g \otimes * \rho_2(*g)))\}. 
\end{eqnarray*} 
We express the action of $PL(2)_v$ on $Q$ as $u \cdot g = u g$ for $u \in Q$ and $g \in PL(2)_v$.  
Then each structure group acts on the bundle of each Cartan connection as follows:   
\begin{eqnarray*} 
&(2)& u \cdot * g = u g,  \\ 
&(3)& (u, A) \cdot (*g, \rho_2(*g)) = (u g, A \rho_2(*g)), \\
&(4)& (u, A) \cdot (g, *\rho_2(*g)) = (u g, A \rho_2(*g)), \\
&(5)& (u, A, B) \cdot (g, *\rho_2(*g), \rho_5(g \otimes *\rho_2(*g))) = (u g, A \rho_2(*g), 
B \rho_5(g \otimes * \rho_2(*g))).  
\end{eqnarray*}

Successive castling transformations yields a sequence of manifolds admitting a projective structure or   
a Grassmannian structure.  From now on we characterize those manifolds. 

\smallskip 
\begin{lem}\label{castling lemma about isotropy}
Let $v$ be a point in $P(\R^l \otimes \R^\alpha)$ and 
assume that $PL(l) \times PL(\alpha) \cdot v$ gives an open orbit in $P(\R^l \otimes \R^\alpha)$. 
Suppose that a Lie group $G$ is obtained by successive castling 
transformations from $PL(l) \times PL(\alpha)$, and a point $w$ in $P(\R^l \otimes \bigotimes_{i=1}^j\R^{k_i})$ 
is obtained from $v$.         
Then $G$ can be written as $PL(l) \times \prod_{i=1}^j PL(k_i)$ and 
$G \cdot w$ gives an open orbit in $P(\R^l \otimes \bigotimes_{i=1}^j \R^{k_i})$ again. 
Moreover there exists a projective linear representation $f_i$: $PL(l)_{<v>} \to PL(k_i)$ \ $(0 \leq i \leq j)$ 
such that $G_w = \{(f_0(g), f_1(g), \ldots, f_j(g)) \mid g \in PL(l)_{<v>}\}$, where $f_0 = id$ or $*$. 
\end{lem} 
\begin{proof} 
First about the group $PL(l) \times PL(\alpha)$, the element of the isotropy subgroup $PL(l) \times PL(\alpha)_v$ is 
expressed as $(g, \rho_\alpha(g))$, where $\rho_\alpha$ is the projective linear representation of $PL(l)$  
introduced 
after Proposition \ref{cast of rep}. 
We now proceed by induction. 
Assume that there exists $w \in P(\R^l \otimes \bigotimes_{i=1}^j\R^{k_i})$ such that 
the group $PL(l) \times \prod_{i=1}^j PL(k_i)$ admits an open orbit given by $w$, and  
an element of the isotropy group $PL(l) \times \prod_{i=1}^j PL(k_i)_w$ is expressed as 
$(f_0(g), f_1(g), \ldots, f_j(g))$ for some $g \in PL(l)_{<v>}$.  Then the point $w$ must belong to $V_{k_j, l k_1 \cdots k_{j-1}}$.  
A group obtained by using the assertion (1) or (2) in Proposition \ref{successive castling transf of rep} 
also admits an open orbit and its isotropy group is described by using the projective linear representations of 
$PL(l)_{<v>}$. 
Since successive castling transformation consists of Propositions \ref{successive castling transf of rep} and 
\ref{cast of rep}, it is enough to consider the effect of castling transformation described in Proposition 
\ref{cast of rep}. 
By castling transformation of the group $PL(l) \times \prod_{i=1}^j PL(k_i)$, 
we obtain the group 
$PL(l) \times \prod_{i=1}^{j-1} PL(k_i) \times PL(l k_1 \cdots k_{j-1} - k_j)$ and a fixed pt 
$w^\perp \in V_{l k_1 \cdots k_{j-1} - k_j, l k_1 \cdots k_{j-1}}$, which gives an open orbit in 
$P(\R^l \otimes \bigotimes_{i=1}^{j-1} \R^{k_i} \otimes \R^{l k_1 \cdots k_{j-1} - k_j})$.  
Here we regard $PL(l) \times \prod_{i=1}^{j-1} PL(k_i)$ as the subgroup of 
$PL(\R^l \otimes \bigotimes_{i=1}^{j-1} \R^{k_i})$.   
Then about the isotropy group  
$PL(l) \times \prod_{i=1}^{j-1} PL(k_i) \times PL(l k_1 \cdots k_{j-1} - k_j)_{w^\perp}$ 
its element is expressed as $(* f_0(g), * f_1(g), \ldots, * f_{j-1}(g), 
\rho_{l k_1 \cdots k_{j-1} - k_j}(* f_0(g) \otimes * f_1(g) \otimes \cdots \otimes * f_{j-1}(g)))$ for some 
$g \in PL(l)_{<v>}$, 
where $\rho_{l k_1 \cdots k_{j-1} - k_j}$ is the projective linear  
representation of $PL(\R^l \otimes \bigotimes_{i=1}^{j-1} \R^{k_i})_{<w^\perp>}$.  
Thus we completes the induction step. 
\end{proof}
\vspace{-1mm}
\noindent
This Lemma also shows the fact that $G_w$ is isomorphic to $PL(l)_{<v>}$. 

\medskip
Now let $(Q, \omega)$ be a Grassmannian Cartan connection of type $(l-\alpha, \alpha)$ over $M$. 
Thus the model space is $PL(l)/{PL(l)_{<v>}}$, where $<\!v\!>$ is an element of $Gr_{\alpha, l}$.  
By successive castling transformations from $(Q, \omega)$ 
we obtain a Cartan connection $(Q \times \prod_{i=1}^j PL(k_i), \omega' \times \prod_{i=1}^j  (\Lambda_1)_i)$ of type $PL(l) \times \prod_{i=1}^j PL(k_i)/$ $PL(l) \times \prod_{i=1}^j PL(k_i)_w$,  
where $(\Lambda_1)_i = \Lambda_1$ or $*\Lambda_1$ and $\omega' = \omega$ or    
$\omega' = *\omega$.  We denote by $N$ the base space of $(Q \times \prod_{i=1}^j PL(k_i), \omega' \times \prod_{i=1}^j  (\Lambda_1)_i)$.  
Then the Cartan connection $(Q \times \prod_{i=1}^j PL(k_i), \omega' \times \prod_{i=1}^j   (\Lambda_1)_i)$ over $N$ 
induces a projective Cartan connection over $N$. The next proposition determines the relation of base spaces 
obtained by successive castling transformations. We assume that $k_i \neq 1$. 
Remove the $s$-th component from $(Q \times \prod_{i=1}^j PL(k_i), \omega' \times \prod_{i=1}^j   (\Lambda_1)_i)$
and denote it by $(Q \times \prod_{i \neq s}^j PL(k_i), 
\omega' \times \prod_{i \neq s}^j  (\Lambda_1)_i)$. 

\newpage
\begin{prop}\label{base space of successive Cartan conne} 
Choose $s$ and $t$ satisfying $1 \leq s < t \leq j$.  

\noindent
$(1)$ The base space $N$ 
is a principal fiber bundle over $M$ with group $\prod_{i=1}^j PL(k_i)$. 

\smallskip
\noindent
$(2)$ 
The pair $(Q \times \prod_{i \neq s}^j PL(k_i), 
\omega' \times \prod_{i \neq s}^j  (\Lambda_1)_i)$ \mbox{\rm(}resp. $(Q \times \prod_{i \neq t}^j PL(k_i), 
\omega' \times \prod_{i \neq t}^j  (\Lambda_1)_i)$ \mbox{\rm)} is  a 
Cartan connection on a manifold $L$ \mbox{\rm(}resp. $K$\mbox{\rm)}, which is a subgeometry of a Grassmannian Cartan connection of type $(l k_1 \cdots k_{s-1} k_{s+1} \cdots k_j - k_s, k_s)$ 
\mbox{\rm(}resp. $(l k_1 \cdots k_{t-1} k_{t+1} \cdots k_j - k_t, k_t)$\mbox{\rm)}.  
The base space $L$ is a $\prod_{i \neq s}^j PL(k_i)$-bundle over $M$. 

\smallskip
\noindent
$(3)$   
The base space $N$ is regarded as a principal fiber bundle over $L$ \mbox{\rm(}resp. $K$\mbox{\rm)} with group $PL(k_s)$ 
\mbox{\rm(}resp. $PL(k_t)$\mbox{\rm)}, and 
$PL(k_t)$ acts on $N$ and $L$. 
Thus we obtain following diagram:   
{\small
\[
	\xymatrix{\SelectTips{xy}{10}
	 & PL(k_s) \ar[r] & N \ar[dll] \ar[drr]  & PL(k_t) \ar[l] &  \\ 
	  L  &   PL(k_t) \ar[l] & &  PL(k_s)\ar[r] & K. \\
    } 
\]
}
Moreover $N$ is isomorphic to the bundle 
$P_tL/GL(l k_1 \cdots k_{s-1} k_{s+1} \cdots k_j - k_s) \otimes GL(1)$ and the action of $PL(k_t)$ on $L$ induces the action on $N$ by the differential. The quotient $N/PL(k_t)$ is isomorphic to $K$.   
\end{prop}
\vspace{-1mm}  
\begin{proof} 
We can assume that the number $t$ is equal to $j$ without loss of generality. 
Firstly we show that $N$ is a principal fiber bundle over $M$. 
From the assumption $N$ is diffeomorphic to the quotient manifold 
$Q \times \prod_{i=1}^j PL(k_i)$/$PL(l) \times$ $\prod_{i=1}^j PL(k_i)_w$, 
thus we identify $N$ with this quotient manifold. 
By using this identification the group $\prod_{i=1}^j PL(k_i)$ naturally acts on $N$ as follows: 
let $[z, A_1$, $\ldots, A_j]$ be an element of $N$ and $(B_1, \ldots, B_j)$ be an element of $\prod_{i=1}^j PL(k_i)$. We define the action by $[z, A_1, \ldots, A_j] \cdot (B_1, \ldots, B_j) := [z, B_1^{-1} A_1, \ldots, B_{j}^{-1}A_j]$. This action is free. 
By using the projection $\pi:Q \to M$  
we define the projection $\pi_N: N \to M$ by $[z, A_1, \ldots, A_j] \mapsto \pi(z)$. 
Moreover $\pi:Q \to M$ is a  principal fiber bundle, thus for each open neighborhood $U$ of $M$ there is a local trivialization 
$\pi^{-1}(U) \to U \times PL(l)_{<v>}$  mapping $z$ to $(\pi(z), \phi(z))$. 
Now by Lemma \ref{castling lemma about isotropy} the isotropy subgroup ${PL(l) \times \prod_{i=1}^j  PL(k_i)_w}$ is given by the set 
$\{(f_0(g), f_1(g), \ldots, f_j(g)) \mid g \in PL(l)_{<v>}\}$. 
We define the map 
$\Phi: \pi_N^{-1}(U) \to U \times \prod_{i=1}^j PL(k_i)$ by 
\vspace{-2mm}
\[[z, A_1, \ldots, A_j] \mapsto (\pi(z), {f_1}'(\phi(z))  A_1^{-1}, \ldots, {f_j}'(\phi(z)) A_j^{-1}),\]  
where ${f_i}' = f_i$ when $(\Lambda_1)_i = \Lambda_1$ or 
${f_i}' = *f_i$ when $(\Lambda_1)_i = *\Lambda_1$. 
The map $\Phi$  is well-defined and diffeomorphism, moreover preserving the action of $\prod_{i=1}^j PL(k_i)$. 
By using $\pi_N$ and the local trivializations $\Phi$, it is shown that $N$ is a principal fiber 
bundle over $M$ with group $\prod_{i=1}^j PL(k_i)$.  

Since $\prod_{i=1}^{j-1} \{e\} \times PL(k_j)$ is a normal closed subgroup of the structure group 
$\prod_{i=1}^j PL(k_i)$ of $N$,  the quotient $N/{PL(k_j)}$ is again a principal fiber bundle over $M$ with group 
$\prod_{i=1}^{j-1} PL(k_i)$.  Put $H = \prod_{i=1}^{j-1} PL(k_i)$. Then 
the manifold $Q \times H$ is regarded as a principal fiber 
bundle over $M$ with group $PL(l)_{<v>}' \times H$, where $PL(l)_{<v>}' = PL(l)_{<v>}$ if 
$\omega' = \omega$ or $PL(l)_{<v>}' = *PL(l)_{<v>}$ if $\omega' = *\omega$. 
The group $PL(l)_{<v>}' \times H$ 
contains the closed subgroup 
$PL(l) \times H_{<w>}$. 
Hence we obtain the fiber bundle $Q \times H$ over the quotient manifold 
$Q \times H/{PL(l) \times H_{<w>}}$ with structure group 
$PL(l) \times H_{<w>}$. 
There is a diffeomorphism from  $Q \times H/PL(l) \times H_{<w>}$ 
to $N/PL(k_j)$ defined by 
$[u,A_1,\ldots, A_{j-1}] \mapsto [u,A_1,\ldots, A_{j-1}, e]PL(k_j)$. By Proposition \ref{cast of geo} 
$(Q \times H, \omega \times \prod_{i=1}^{j-1} \Lambda_1)$ is a Cartan connection over 
$N/PL(k_j)$ of type 
$PL(l) \times H/PL(l) \times H_{<w>}$. 
From the assumption the model space 
$PL(l) \times H/PL(l) \times H_{<w>}$  
is a subgeometry of the Grassmannian manifold 
$PL(\R^l \otimes \bigotimes_{i=1}^{j-1} \R^i)/PL(\R^l \otimes \bigotimes_{i=1}^{j-1} \R^i)_{<w>}$ according to the map   
$F: PL(l) \times H \hookrightarrow PL(\R^l \otimes \bigotimes_{i=1}^{j-1} \R^i)$ 
defined by the tensor product.    
Note that a point $w$ is included in $V_{k_j, l k_1 \cdots k_{j-1}}$.  
Therefore $N/PL(k_j)$ admits a Grassmannian Cartan connection of type $(l k_1 \cdots k_{j-1} - k_j, k_j)$. 
\vspace{2mm}
\newline 
\indent We observe naturally $PL(k_s)$ acts freely on $Q \times H$ by 
$(u, A_1, \ldots, A_s$, $\ldots, A_{j-1})$ $\cdot B_s$ $:=$ $(u, A_1, \ldots, B_s^{-1} A_s, \ldots, A_{j-1})$ for $B_s \in PL(k_s)$. We denote this right action by $R_{B_s}$. 
Put $L = N/PL(k_j)$. The action of $PL(k_s)$ on $Q \times \prod_{i=1}^{j-1} PL(k_i)$ induces the action of $PL(k_s)$ on $L$ by using the bundle isomorphism between $L$ and $Q \times H/
{PL(l) \times H_{<w>}}$. 
\vspace{2mm}
\newline 
\indent We fix the complementary subspace $\m$ of $\sll(l)_{<v>}$ in $\sll(l)$, thus we have 
$\sll(l) = \m \oplus \sll(l)_{<v>}$. Then 
$\m' \times \prod_{i=1}^{j-1} \sll(k_i)$ gives a complementary subspace of 
$\sll(l) \times \prod_{i=1}^{j-1} \sll(k_i)_{<w>}$ in $\sll(l) \times \prod_{i=1}^{j-1} \sll(k_i)$, where 
$\m' = \m$ when $\omega' = \omega$ or $\m' = *\m$ when $\omega' = *\omega$.  Furthermore 
the differential $dF$ of $F$ induces a linear isomorphism $\hat{dF}$ from 
$\m' \times \prod_{i=1}^{j-1} \sll(k_i)$ to $M(l k_1 \cdots k_{j-1} - k_j, k_j)$. 
\vspace{2mm}
\newline 
\indent 
We denote by $\rho$ the isotropy representation of the model space 
$PL(l) \times H/PL(l) \times H_{<w>}$. 
Denote $Q \times H$ by $P$, and the natural projection 
$P \to P/ker\rho$ by $\rho$. We denote by $\widetilde{P}$ the quotient space $P/ker\rho$. 
The action of $PL(k_s)$ on $P$ induces the action on 
$\widetilde{P}$. 
Then there exists a unique 
1-form $\theta$ on $\widetilde{P}$ such that 
$\rho^* \theta$ $=$ 
${\omega' \times \prod_{i=1}^{j-1} \Lambda_1}_{\m' \times \prod_{i=1}^{j-1} \sll(k_i)}$.  
The pair $(\widetilde{P}, \theta)$ gives a $\rho(PL(l) \times H_{<w>})$-structure  
over $L$, which is a subbundle of a $GL(l k_1 \cdots k_{j-1} - k_j) \otimes GL(k_j)$-structure $P_tL$.  
By Proposition \ref{red of G-structure} 
the imbedding $\iota$ of $\widetilde{P}$ into $P_tL$ is given by the restriction of the bundle isomorphism 
$t: \mathcal{L}(L) \to \mathcal{L}(L)$, which is defined by $t: x \mapsto x \circ \hat{dF}^{-1}$.  
\vspace{2mm}
\newline
\indent 
The equality ${R_{B_s}}^* \omega \times \prod_{i=1}^{j-1} \Lambda_1$ $=$ $\omega \times \prod_{i=1}^{j-1} \Lambda_1$ 
yields  
${R_{B_s}}^* \rho^* \theta = \rho^* \theta$. 
Since we have the commutative diagram 

\begin{equation} \label{commutative diagram 5}
\xymatrix{
	 P \ar[r]^{R_{B_s}} \ar[d]_\rho \ar@{}[dr]|\circlearrowleft &  P \ar[d]_{\rho} \\
     \widetilde{P} \ar[r]^{R_{B_s}} \ar[d] \ar@{}[dr]|\circlearrowleft & \widetilde{P} \ar[d] \\	
     L    \ar[r]^{R_{B_s}}    &  L,  
     }
\end{equation} 
it follows that ${R_{B_s}}^* \rho^* \theta = \rho^* {R_{B_s}}^*\theta$. Thus ${R_{B_s}}^* \theta = \theta$.  
Therefore the action $R_{B_s}: \widetilde{P} \to \widetilde{P}$ is induced by the differential of the 
action $R_{B_s}: L \to L$. Moreover $R_{B_s}: \tilde{P} \to \tilde{P}$ is 
uniquely extended to the action $R_{B_s}:P_tL \to P_tL$  by 
\[\iota \circ \rho(z, A_1, \ldots, A_s, \ldots, A_{j-1}) B \otimes C \cdot D 
= \iota \circ \rho(z,A_1, \ldots, D^{-1} A_s, \ldots, A_{j-1}) B \otimes C, \] 
where $(z, A_1, \ldots, A_s, \ldots, A_{j-1}) \in Q'$,  $B \otimes C$ $\in$ $GL(l k_1 \cdots k_{j-1} - k_j) 
\otimes GL(k_j)$ and $D \in PL(k_s)$. 
This action naturally induces the action of $PL(k_s)$ on $P_tL/GL(lk_1\cdots k_{j-1}-k_j) \otimes GL(1)$.  
We describe the process that the action of $PL(k_s)$ on $P$ induces the actions on other manifolds 
by the following diagram. 
\[ \xymatrix{ 
    PL(k_s) \ar@/^1pc/[d] &  PL(k_s) \ar@/^1pc/[d] & P L(k_s) \ar@/^1pc/[d] & PL(k_s) \ar@/^1pc/[d] \\            
    P \ar@{-->}[r] & \widetilde{P} \ar@{-->}[r] & P_tL \ar@{-->}[r] & P_tL/GL(lk_1\cdots k_{j-1}-k_j) \otimes GL(1) 
   }
\] 
By the proof of Proposition \ref{cast of geo} there is a $PL(k_j)$-bundle isomorphism from $N$ $=$ 
$P \times PL(k_j)/H \times \prod_{i=1}^j PL(k_i)_w$ to $P_tL/GL(lk_1\cdots k_{j-1}-k_j) \otimes GL(1)$. Via this isomorphism we see that  
the action of $PL(k_s)$ on $N$ coincides with the one on $P_tL/GL(lk_1\cdots k_{j-1}-k_j) \otimes GL(1)$. 
Thus the induced action of $PL(k_s)$ on $N$ is given by 
$[u, A_1, \ldots, A_s, \ldots, A_j] \cdot B_s = [u, A_1, \ldots, B_s^{-1}A_s, \ldots, A_j]$ for $B_s \in PL(k_s)$. 
Consequently the action of $PL(k_s)$ on $L$ induces the action on $N$ by the differential.   
By this action we have the quotient $K = N/PL(k_s)$. On the other hand 
let $\sigma$ be a permutation defined by 
\[\sigma = 
\begin{pmatrix} 
1 & 2 & \cdots & s-1 & s & s+1 & \cdots & j-1 & j \\ 
1 & 2 & \cdots & s-1 & s+1 & \cdots & j-1 & j & s 
\end{pmatrix}.  
\] 
Then by Proposition \ref{successive castling transf of Cartan} 
the given Cartan connection $(Q \times \prod_{i=1}^{j} PL(k_i), \omega' \times \prod_{i=1}^{j} (\Lambda_1)_i)$ over  $N$ can be regarded as a Cartan connection 
over $N$ of type $PL(l) \times \prod_{i=1}^{s-1} PL(k_i) \times \prod_{i=s+1}^{j} PL(k_i) \times PL(k_s)/PL(l) \times  \prod_{i=1}^{s-1} PL(k_i) \times \prod_{i=s+1}^{j}$ $PL(k_i)\times PL(k_s)_{\sigma(w)}$, 
which is a subgeometry of the projective space  
$P(\R^l \otimes \bigotimes_{i=1}^{s-1} \R^{k_i} \otimes \bigotimes_{i = s+1}^{j} \R^{k_i} \otimes \R^s)$.  
By Proposition \ref{cast of geo} it follows that $K$ admits a Grassmannian Cartan connection of type 
$(l k_1 \cdots k_{s-1} k_{s+1} \cdots k_{j} - k_s, k_s)$. 
\end{proof} 

\section{Examples of successive castling transformations}  
Let $(Q, \omega)$ be a projective Cartan connection over $M$.   
As we demonstrate it after Proposition \ref{successive castling transf of Cartan}
by successive castling transformations we can obtain the following Cartan connections:  
\begin{eqnarray*} 
&& \hspace{-0.2cm} \ (Q, \omega) \\ 
&& \hspace{-0.5cm} \longrightarrow 
(Q \times PL(n), *\omega \times \Lambda_1) \longrightarrow 
(Q \times PL(n) \times PL(n^2 + n -1), \omega \times *\Lambda_1 \times \Lambda_1) \\ 
 && \hspace{-0.5cm} \longrightarrow (Q \times PL(n^2 + n -1), \omega \times \Lambda_1). 
\end{eqnarray*}
We denote the base space of $(Q, \omega)$ by $M_1$, the one of 
$(Q \times PL(n), *\omega \times \Lambda_1)$ by $M_n$ and so on. 
Generally 
$M_{k_1 \times \cdots \times k_j}$ denotes a principal fiber bundle over $M$ with structure 
group $\prod_{i=1}^j PL(k_i)$. 
Then following the above successive Cartan connections from $(Q, \omega)$ we obtain the sequence of base spaces: 
$M_1 \longrightarrow M_2 \longrightarrow M_{2 \times 5} \longrightarrow M_5$.
The Cartan connection
$(Q \times PL(n), *\omega \times \Lambda_1)$ over $M_n$ and  
$(Q \times PL(n) \times PL(n^2 + n -1), \omega \times *\Lambda_1 \times \Lambda_1)$ over 
$M_{n \times (n^2 + n -1)}$  
induce projective Cartan connections, and the Cartan connection  
$(Q \times PL(n^2 + n - 1), \omega \times \Lambda_1)$ over $M_{n^2 + n -1}$ induces a Grassmannian Cartan 
connection. 
Now we describe those base spaces more explicitly: 

\medskip
\begin{prop}\label{example of base space} 
$M_n$ is isomorphic to the projective frame bundle of $M$, and $M_{n \times (n^2 + n -1)}$ and $M_{n^2 + n -1}$ are isomorphic to the following bundles respectively:  
\vspace{-1mm}
\begin{eqnarray*}
\widetilde{\mathcal{L}}_{n \times (n^2 + n -1)} := \hspace{-5mm} 
&& \{ (u_1, u_2) \mid u_1: \mbox{projective frame of} \ T_pM, \\ 
&& \ u_2: \mbox{projective frame of} \ T_pM \rtimes \sll(T_pM), p \in M \}, \\
\widetilde{\mathcal{\mathcal{L}}}_{n^2 + n -1} := \{u_2 \mid && \hspace{-5mm} u_2: \mbox{projective frame of} \ T_pM \rtimes \sll(T_pM), p \in M \}.
\end{eqnarray*} 
\end{prop}
\begin{proof} 
By definition $Q$ is a principal fiber bundle over $M$ with structure group $PL(n)_v$, where  
$v = (1,0, \cdots, 0)$ is an element of $V_{1,n}$. From the argument of $\S$ 2.1  
we have the injection 
$h: \mathcal{L}(M) \hookrightarrow Q$ corresponding to the injection $\iota: GL(n) \hookrightarrow PL(n+1)_v$, which is defined by 
\[\iota: A \mapsto 
\begin{pmatrix}
1 & 0 \\
0 & A 
\end{pmatrix}.  
\]  
By castling transformation of $(Q, \omega)$ we obtain the Cartan connection 
$(Q \times PL(n), *\omega \times \Lambda_1)$ 
 over $M_n$ whose structure group is $PL(n+1) \times PL(n)_{<v^\perp>}$, where 
$v^\perp$ $=$  
$
\begin{pmatrix}
0 \\
\hline
I_n
\end{pmatrix} 
$ 
is an element of $V_{n, n+1}$. 
Then $\iota(GL(n)) \times PL(n)_{v^\perp}$ is given by the set $\{(*\iota(A), {A})\}$. 
The manifold $h(\mathcal{L}(M)) \times PL(n)$ is a principal fiber bundle over $M_n$ with structure group 
$\iota(GL(n)) \times PL(n)_{v^\perp}$, and gives a reduction of $Q \times PL(n)$.  
Then we have the following bundle isomorphism 
$h(\mathcal{L}(M)) \times PL(n)/\iota(GL(n)) \times PL(n)_{v^\perp}$ 
$\to$ 
$\mathcal{L}(M)/GL(1)$ defined by    
\[ 
[h(x), g] \mapsto q(x) g^{-1},   
\]
where $q$ is the projection $\mathcal{L}(M) \to \mathcal{L}(M)/GL(1)$. 
Note that the action of $\iota(GL(n)) \times PL(n)_{v^\perp}$ on $h(\mathcal{L}(M)) \times PL(n)$ is given by   
$(h(x), g) \cdot (*\iota(A), {A})$ $=$ $(h(x)\iota(A), g {A})$.  
Thus the base space $M_n$ is isomorphic to the projective frame bundle of $M$.  
Concerning the Cartan connection 
$(Q \times PL(n) \times PL(n^2 + n -1), \omega \times *\Lambda_1 \times \Lambda_1)$ over 
$M_{n \times (n^2 + n -1)}$, 
the structure group of $Q \times PL(n) \times PL(n^2 + n -1)$ is 
$PL(n+1) \times PL(n) \times PL(n^2 + n -1)_{w^\perp}$, where  
$w^\perp$ is a fixed projective frame of the vector space 
$ 
\begin{pmatrix}
M(1,n) \\ 
\hline 
\sll(n) 
\end{pmatrix}. 
$ 
When we are given a base $x$ of a vector space $V$, we denote by $q(x)$ the projective frame, 
then $q$ gives the projection from the linear Stiefel manifold to the projective Stiefel manifold corresponding to 
the natural projection $GL(n) \to PL(n)$. 
Then the subgroup $\iota(GL(n)) \times PL(n) \times PL(n^2 + n -1)_{w^\perp}$ is given by the set 
$\{(\iota(A), *{A}, \rho_{n^2 + n -1}(\iota(A) \otimes *{A} ) )\}$, where $\rho_{n^2 + n -1}(\iota(A) \otimes *{A})$ 
is expressed as the matrix 
\[
{
\begin{pmatrix} 
A & 0 \\ 
0 & Ad(A) 
\end{pmatrix}}
\]   
with respect to the basis 
\[
\{ 
\begin{pmatrix} 
e_i \\ 
\hline 
0 
\end{pmatrix} 
 (1 \leq i \leq n), 
\begin{pmatrix}  
0 \\ 
\hline 
E_j^k - \delta_j^k E_n^n  
\end{pmatrix} 
 (1\leq  i,j \leq n) 
\}. 
\]  
By using a base $x = \{X_1, \cdots, X_n\}$ of $T_pM$ and $Y = \sum Y^i_j E_i^j \in \sll(n)$ we define an 
element $Y_{x} \in \sll(T_pM)$ by $Y_{x} :=  x \circ Y \circ x^{-1}$, where we regard $x$ as a linear isomorphism.  
That is to say $Y_x$ is the map  
$Y_{x}:X_k \mapsto \sum Y_k^i X_i$.    
Then for $A \in GL(n)$ we have $Y_{x A} = Ad(A)(Y_x)$. We denote by ${E'}_j^k$ the element 
$E_j^k - \delta_j^k E_n^n$ of $\sll(n)$. 
Now we define the map $\phi: h(\mathcal{L}(M)) \times PL(n) \times PL(n^2 + n -1)$ $\to$ 
$\widetilde{\mathcal{L}}_{n \times (n^2 + n -1)}$ by 
\[
[h(x), g_1, g_2] \mapsto \{ q(x) {g_1}^{-1}, q(x, ({E'}_j^k)_x){g_2}^{-1})\} 
\] 
We show that the map $\phi$ is well defined.  
We put 
\[(h(x'), g_1', g_2') := (h(x), g_1, g_2) \cdot (\iota(A), *{A}, \rho_{n^2 + n -1} 
(\iota(A) \otimes  *{A})),\]  
which is equal to $(h(x A), g_1 {A}, g_2 \rho_{n^2 + n -1}(\iota(A) \otimes *{A}))$. 
Then we have 
\begin{eqnarray*} 
&& \{q(x') {g_1'}^{-1}, q(x', ({E'}_j^k)_{x'}){g_2'}^{-1})\} \\ 
&=& \{q(x A) (g_1 {A})^{-1}, q(x A, ({E'}_j^k)_{x A})
(g_2 \rho_{n^2 + n -1}(\iota(A) \otimes *{A}))^{-1} \}. 
\end{eqnarray*} 
The last expression is equal to $\{q(x) {g_1}^{-1}, q(x, ({E'}_j^k)_x){g_2}^{-1})\}$ 
since we have 
\[q(x A, ({E'}_j^k)_{x A}) = q(x, ({E'}_j^k)_x) \rho_{n^2 + n -1}(\iota(A) \otimes *{A}).\]   
Moreover $\phi$ is a bundle isomorphism.  
Likewise we can show that $M_{n^2 + n -1}$ is isomorphic to $\widetilde{\mathcal{L}}_{n^2 + n -1}$.  
\end{proof}  

Now we explain how these manifolds $M_{k_1 \times \cdots \times k_j}$ are related in the case of $n=2$. 
From Propositions \ref{example of base space} and \ref{cast of geo} 
$M_2$ is a projective frame bundle of $M_1$ and $M_{2 \times 5}$ is a 
projective frame bundle of $M_2$. Thus $M_2$ has the right action of $PL(2)$ and 
this action gives rise to the action of $PL(2)$ on the frame bundle of $M_2$ by the differential. 
Thus $PL(2)$ naturally acts on the projective frame bundle $M_{2 \times 5}$ of $M_2$. 
From Proposition \ref{base space of successive Cartan conne} the quotient $M_{2 \times 5}/PL(2)$ is equal to 
$M_{5}$.  
Furthermore Cartan connections $(Q, \omega)$, $(Q \times PL(2), *\omega \times \Lambda_1)$ and 
$(Q \times PL(2) \times PL(5), \omega \times *\Lambda_1 \times \Lambda_1)$ induce projective Cartan connections and $(Q \times PL(5), \omega \times \Lambda_1)$ induces Grassmannian Cartan connections. 
The same result holds true about the general dimension $n$. 
If we continue the successive castling transformations we can obtain the following tree, where    
we only describe the base spaces. We abbreviate 
$M_{k_1 \times \cdots \times k_j}$ to $k_1 \times \cdots \times k_j$. 
If a Cartan connection over a base space induces a Grassmannian structure of type $(\beta, \alpha)$ then 
we write $GL(\beta) \otimes GL(\alpha)$ under the base space. If a Cartan connection induces a projective structure,  then we write nothing.  
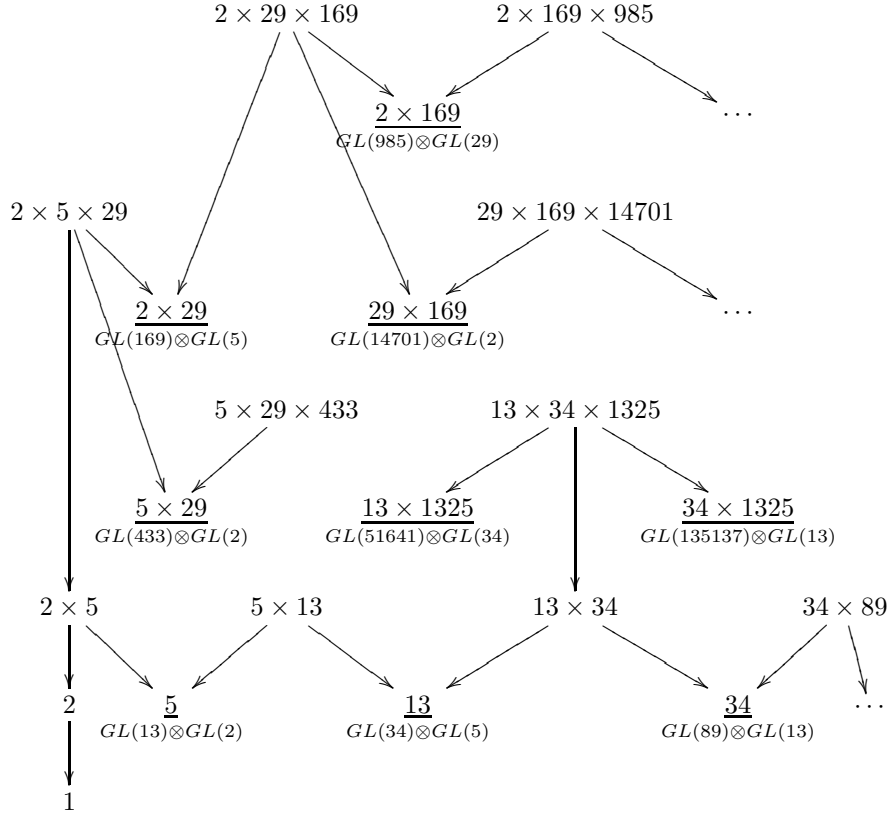
\[  \xymatrix@C-1.5cm@R-2mm{  
    && {{{\color{black} 2 \times 29 \times 169}}} \ar[lddd]  \ar[rddd] \ar[rd] && 
      {{{\color{black} 2  \times 169 \times 985 }}} 
      \ar[ld]  \ar[rd]      &&&           \\
      &&& \underset{GL(985) \otimes GL(29)}{{
    \underline{{\color{black}2 \times 169}}}} && \cdots && \\ 
     {{{\color{black}2 \times 5 \times 29}}} \ar[dddd] \ar[rddd] \ar[rd] & & & & 
     {{{\color{black} 29 \times 169 \times 14701 }}}  \ar[ld] \ar[rd]     & && \\
	& \underset{GL(169) \otimes GL(5)}{{\underline{{\color{black}2 \times 29}}}} && 
	\underset{GL(14701)  \otimes GL(2)}{{\underline{{\color{black} 29 \times 169}}}}  && \cdots &&\\
	&& {{{\color{black}5 \times 29 \times 433}}} \ar[ld] & & 
	{{{\color{black}13 \times 34 \times 1325}}} \ar[dd] \ar[rd] \ar[ld]    & && \\
	& \underset{GL(433) \otimes GL(2)}{{\underline{{\color{black}5 \times 29}}}} && 
	\underset{GL(51641)  \otimes GL(34)}{{\underline{{\color{black} 13 \times 1325}}}}  & & 
    \underset{GL(135137)  \otimes GL(13)}{{\underline{{\color{black} 34 \times 1325}}}} && \\ 
    {{{\color{black}2 \times 5}}} \ar[d] \ar[rd] & & 
    {{{\color{black}5 \times 13}}} \ar[ld] \ar[rd] & & 
    {{{\color{black}13 \times 34}}} \ar[ld] \ar[rd] 
    & &  {{{\color{black} 
    34 \times 89}}} \ar[ld] \ar[rd] & \\ 
    {{{\color{black}2}}} \ar[d] & 
     \underset{GL(13) \otimes GL(2)}{{\underline{{\color{black}5}}}} & & 
    \underset{GL(34) \otimes GL(5)}{{\underline{{\color{black}13}}}} & & 
    \underset{GL(89) \otimes GL(13)}{{\underline{{\color{black}34}}}}  
    & & \cdots \\ 
	{{{\color{black}1}}} &&&&&&& \\
    }
\] 
The relation of the base spaces of the tree is completely described by Proposition \ref{base space of successive Cartan conne}. 
The all underlined manifolds  
admit a Grassmannian structure.  If the given projective structure on $M$ is projectively flat, then 
manifolds which are not underlined  
admit 
a flat projective structure and underlined manifolds admit a flat Grassmannian structure.   
Especially a underlined manifold $M_{k_1 \times \cdots \times k_{j-1}}$  
admits a Grassmannian structure of 
type $(\beta, \alpha)$ 
which is given by an extension of the bundle $\mathcal{L}(M) \times \prod_{i=1}^{j-1} PL(k_i)$ over 
$M_{k_1 \times \cdots \times k_{j-1}}$.  
The structure group of $\mathcal{L}(M) \times \prod_{i=1}^{j-1} PL(k_i)$ is isomorphic to $GL(2)$. This $GL(2)$-bundle gives a reduction of 
$GL(3 k_1 \cdots k_{j-1} - k_j) \otimes GL(k_j)$-structure 
on $M_{k_1 \times \cdots \times k_{j-1}}$.  
We can prove this assertion generally as follows: we use the same notations in Proposition 
\ref{base space of successive Cartan conne}, thus $M$ admits a Grassmannian Cartan connection of type 
$(l-\alpha, \alpha)$, and $N$ is a fiber bundle $M_{k_1 \times \cdots \times k_j}$ equipped with a projective structure. 
Denote by $L$ the quotient $N/PL(k_j)$ and by $P$ the bundle $Q \times H$ over 
$L$ with structure group $PL(l) \times H_{<w>}$. Now we assume that $l k_1 \cdots k_{j-1} -k_j \neq 1$, hence 
$L$ admits a Grassmannian structure $P_tL$.  
The Lie group homomorphism $\rho: PL(l) \times H_{<w>} \to GL(M(l-\alpha, \alpha) \times  \prod_{i=1}^{j-1} \sll(k_i))$ is the isotropy representation of 
$PL(l) \times H/ PL(l) \times H_{<w>}$. 
The quotient space $(P/ker\rho, \theta)$ can be considered as a subbundle of $P_tL$, and we have the natural projection $\rho: P \to \widetilde{P}$. The group $PL(l)_{<v>}$ has the subgroup $G_0$ and 
the restriction of the isotropy representation $\rho$ to $G_0 \times H_{<w>}$ is injective. 
The bundle $P$ has the subbundle $h(P_tM) \times H$ with structure group 
$G_0 \times H_{<w>}$.  
Hence the restriction of $\rho: P \to \widetilde{P}$ to $h(P_tM) \times H$ is injective. 
Thus we obtain the sequence of reduction $(h(P_tM) \times H, \theta)$ $\subset$ 
$(\widetilde{P}, \theta)$ $\subset$ $P_tL$ corresponding to the sequence $G_0 \times H_{<w>}$ 
$\subset$ $\rho(PL(l) \times H_{<w>})$ 
$\subset$ $GL(l k_1 \cdots k_{j-1} - k_j) \otimes GL(k_j)$. 
Hence $L$ admits a reduction $h(P_tM) \times H$ of the Grassmannian structure $P_tL$ of $L$. 

\medskip
\noindent
{\bf Case of Lie groups} 

\medskip
\noindent 
In the case of Lie groups the base spaces obtained by successive castling transformations are described more explicitly. 
Let $(\mathcal{L}(L), [\chi])$ be a projective structure on a $n$-dimensional Lie group $L$. Then we can construct a projective Cartan connection $(Q, \omega)$, where $Q$ is a principal fiber bundle over $M$ with structure group 
$PL(n+1)_{v}$, and we denote by $h$ a injective bundle map from $P_L$ to $Q$.  
Then by a successive castling transformations of $(Q, \omega)$ we obtain a Cartan connection 
$(Q \times \prod_{i=1}^j PL(k_i), \omega' \times \prod_{i=1}^j(\Lambda_1)_i)$ over a manifold $N$ whose type is a subgeometry of the projective space $P(\R^{n+1} \otimes \bigotimes_{i=1}^j \R^{k_i})$.  
Then the base space is described as follows:

\medskip
\begin{prop}
$N$ is isomorphic to the product 
$L \times \prod_{i=1}^j PL(k_i)$.   
\end{prop} 
\vspace{-2mm}
\begin{proof} 
The proof is similar to the one of Proposition \ref{base space of successive Cartan conne}. 
The base space $N$ is diffeomorphic to the quotient manifold 
$Q \times \prod_{i=1}^j PL(k_i)/PL(l) \times \prod_{i=1}^j PL(k_i)_w$ and by 
Proposition \ref{castling lemma about isotropy}  the isotropy group $PL(l) \times \prod_{i=1}^j PL(k_i)_w$ 
is given by the set $\{(f_0(g), f_1(g), \ldots, f_j(g)) \mid g \in PL(n+1)_{v}\}$, where $f_0 = id$ or $*$.  
Any element of $Q$ is written by the form $h(x)g_0$ for some $x \in \mathcal{L}(L)$ and 
$g_0 \in PL(n+1)_v$. We denote by $\rho$ a projection from $Q$ to $\mathcal{L}(L)$. 
The frame bundle $\mathcal{L}(L)$ is isomorphic 
to the product $L \times GL(n)$ and an element $x \in \mathcal{L}(L)$ is written as $(a(x), A(x))$. 
Now we construct the map $N \to L \times \prod_{i=1}^j PL(k_i)$ defined by 
\[
[h(x)g_0, g_1, \ldots, g_j] \mapsto (a(x), g_1 f_1'(g_0)^{-1} f_1' \circ \imath(A(x))^{-1}, \ldots, 
g_j f_j'(g_0)^{-1} f_j' \circ \imath(A(x))^{-1}),   
\]
where $f_i' = f_i$ when $(\Lambda_1)_i = \Lambda_1$ and $f_i' = *f_i$ when $(\Lambda_1)_i = *\Lambda_1$. 
Then this map is well defined and a bundle isomorphism.  
\end{proof} 

The Cartan connection 
$(Q \times \prod_{i=1}^{j-1} PL(k_i), \omega' \times \prod_{i=1}^{j-1} (\Lambda_1)_i)$ can be extended to 
Grassmannian Cartan connection, and from the proof of 
\ref{base space of successive Cartan conne} its base space $K$ is also given by the product 
$L \times \prod_{i=1}^{j-1} PL(k_i)$.  
Thus the base space obtained by successive castling transformations 
admitting a Grassmannian structure is also given by a product Lie group.  
Furthermore if the given projective structure $(\mathcal{L}(L), [\chi])$ on $L$ is left invariant (resp. flat) under the group action of $L$, then a projective structure on $N$ is also left invariant (resp. flat), and the Grassmannian 
Cartan connection over $K$ is also left invariant (resp. flat). 

\section{A classification of manifolds obtained by successive \\ castling transformations} 
Let $j, l$ and $\alpha$ be positive natural numbers such that $\alpha \leq  l - \alpha$. 
We consider the equation $(\ast): (l-\alpha)\alpha + k_1^2 + k_2^2 + \cdots + k_j^2 - j =  l k_1 \cdots k_j -1 $. 
We define a castling transformation for a set of positive natural numbers. 

\medskip 
\begin{df} 
Let $l$ be a fixed integer such that $l \geq 3$.  
Let $(k_1, k_2, \ldots, k_j)$ be a set of integers. 
We define an integer by $k_i' := l k_1 \cdots k_{i-1} k_{i+1} \cdots k_j - k_i$ for 
$j \geq 2$, $1 \leq  i \leq j$,  and by $k_{j+1}' := l k_1 \cdots \cdots k_j - 1$. 
If j = 1, we define $k_1'$ to be $l- k_1$. 
We call the set $(k_1, \ldots, k_{i-1}, k_i', k_{i+1}, \ldots, k_j)$ a castling transform of 
$(k_1, k_2, \ldots, k_j)$ at $i$-th position, and $(k_1, \ldots, k_j, k_{j+1}')$ a castling transform 
at $(j+1)$-th position.  We call each $k_i'$ and $k_{j+1}'$ a number obtained by castling transform.    
\end{df} 

\medskip
If $k_i$ is a positive natural number for $1 \leq i \leq j$ and $(k_1,\ldots, k_j)$ satisfies the equation 
$(\ast): (l-\alpha)\alpha + k_1^2 + k_2^2 + \cdots + k_j^2 - j =  l k_1 \cdots k_j -1$, 
then castling transform at any position gives another solution 
$(k_1, \ldots, k_{i-1}, k_i', k_{i+1}, \ldots, k_j)$ for the equation $(\ast)$ and 
$k_i'$ is a positive natural number again.  

We observe that when $j = 1$, $\alpha$ gives a solution for 
the equation $(\ast)$: $(l-\alpha)\alpha + k_1^2 - l k_1 = 0$. We investigate the whole solutions of 
$(\ast)$ given by the successive castling transformations from $\alpha$.  

If we repeat a castling transformation for $(k_1, k_2, \ldots, k_j)$ at the same position twice, then  
we obtain the same set as $(k_1, k_2, \ldots, k_j)$.  From now on we assume that successive castling 
transformations does not include this repetition. Namely  
if a set $\theta := (k_1, k_2, \ldots, k_j)$ is a castling 
transform of a set $\theta'$ at $i$-th position, then 
we only consider a castling transform of $\theta$ at $m$-th position with $m \neq i$.    
A sequence $\theta_1 \rightarrow \cdots \rightarrow \theta_n$ obtained by successive castling transformations 
from $\theta_1$  is said to be reduced if the sequence does not contain any repetition of  castling transformation. 
\medskip 
\begin{lem}\label{the location of cast} 
Let $\theta_1 \rightarrow \cdots \rightarrow \theta_n$ be a reduced sequence obtained by successive castling transformations and 
assume that $\theta_n = (k_1, k_2, \ldots, k_j)$ and $\theta_1 = \alpha$. 
If $k_j$ is a number obtained in the castling transformation $\theta_{n-1} \rightarrow \theta_{n}$,  
then $k_j$ is the unique largest number in $(k_1, k_2, \ldots, k_j)$.  
\end{lem} 
\vspace{-1mm} 
\begin{proof}  
The proof is by induction on $m$ ($n \geq  m \geq 2$). 
We can express $\theta_m$ as $(h_1, h_2, \ldots, h_{N_m})$, where $1 \leq N_m \leq m$ and 
$h_i \geq 2$ ($1\leq  i  \leq N_m$). 
We assume that $h_{N_m}$ is a number obtained by a castling transform of $\theta_{m-1}$ and 
$h_{N_m}$ is the largest number in $\{h_i\}_{1 \leq i \leq N_m}$.      
Then $\theta_{m+1}$ is a castling transform of $\theta_m$ at some $i$-th position with the condition $i \neq N_m$.  
A new number $\kappa$ of $\theta_{m+1}$ obtained by castling transform  
is written as (1) $l h_1 \cdots h_{i-1} h_{i+1} \cdots h_{N_m} - h_i$ or 
(2) $l h_1 \cdots h_{N_m} - 1$, and in both cases $\kappa > h_{N_m}$. The set $\theta_{m+1}$  can be written as $(h_1, \ldots, h_{i-1}, \kappa, h_{i+1}, \ldots, h_{N_m})$ in the case (1) and 
$(h_1, \ldots, h_{N_m}, \kappa)$ in the case (2) . 
Hence $\kappa$ is the unique largest number in $\theta_{m+1}$.    
When $m = 2$, $\theta_2$ can be $(\alpha, l\alpha-1)$ or $l-\alpha$. 
In both cases a number obtained by castling transform is the unique largest number in $\theta_2$. 
Hence the induction proves the lemma. 
\end{proof} 
\begin{prop}\label{fund prop of cast 1} 
Let $(k_1, k_2, \ldots, k_j)$ be a set obtained by successive castling transformations from $\alpha$. 
Then we have $k_i \geq \alpha$ for $1 \leq i \leq j$. 
\end{prop} 
\vspace{-1mm}
\begin{proof}
Let $\theta_1 \rightarrow \cdots \rightarrow \theta_n$ be a reduced sequence obtained by successive castling 
transformations from $\theta_1 = \alpha$.  Denote by $\kappa_i$ ($i \geq 2$) a new number of 
$\theta_i$ obtained 
by the castling transform of $\theta_{i-1}$.  Set $\kappa_1 := \alpha$. 
For instance $\theta_2 = (\kappa_1, \kappa_2) = (\alpha, l\kappa_1-1)$, and 
$\theta_3$ $=$ $(\kappa_1, \kappa_2, \kappa_3)$ $=$ $(\alpha, l\kappa_1-1, l\kappa_1 \kappa_2 - 1)$ 
and $\theta_4$ $=$ $(\kappa_1, \kappa_4, \kappa_3)$ $=$ 
$(\alpha, l\kappa_1\kappa_3 - \kappa_2, l\kappa_1\kappa_2 - 1)$.  
Thus we see that $\theta_m$ ($m \geq 2$) 
can be written as $(\kappa_{i_1}, \kappa_{i_2}, \cdots, \kappa_{i_{N_m}})$ where $N_m \leq m$ and  
some $\kappa_{i_l}$ is equal to $\kappa_m$. 
Then $\kappa_{m+1}$ ($m \geq 2$) is equal to 
$l \kappa_{i_1} \cdots \kappa_{i_{s-1}} \kappa_{i_{s+1}} \cdots \kappa_{i_{N_m}} - \kappa_{i_s}$  for some $s$ 
($1 \leq s \leq N_m$) satisfying $i_s \neq  m$ or  $l \kappa_{i_1} \cdots \kappa_{i_{N_m}} - 1$ since  
successive castling transformations do not include a repetition. 
By Lemma \ref{the location of cast} $\kappa_m$ is the largest number in $\theta_m$, thus $\kappa_{m+1} > \kappa_{m}$ 
for $m \geq 2$. 
By definition $\kappa_2$ can be $l- \alpha$ or $l \alpha -1$.  
From assumption $l- \alpha \geq \alpha$. Thus $\kappa_2 \geq \alpha$. 
Hence $\alpha = \kappa_1 \leq  \kappa_2 < \kappa_3 < \cdots < \kappa_n$, which proves the proposition.  
\end{proof} 

From the proof of this proposition, we see that for any given reduced sequence $\theta_1 \rightarrow \cdots \rightarrow \theta_n$ of successive castling transformations from $\theta_1 = \alpha$,  we have $\theta_i \neq \theta_j$ if $i \neq j$ and $i, j \geq 2$.  
The following is important in this section.  
\begin{prop}\label{fund prop of cast 2} 
Let $j$, $l$ and $\alpha$ be positive natural numbers such that $j \geq 2$, $l \geq 3$ and $\alpha \leq  l - \alpha$. 
Let $k_i$ \mbox{\rm(}$1 \leq i \leq j$\mbox{\rm)} be  natural numbers such that $2 \leq k_1 \leq k_2 \leq \cdots \leq k_j$. 
Assume that we have the equality $(\ast): (l-\alpha)\alpha + k_1^2 + k_2^2 + \cdots + k_j^2 - j =  
l k_1 \cdots k_j -1 $.  Moreover assume that $k_j \geq \alpha$.  
Then we have $0 < l k_1 k_2 \cdots k_{j-1} - k_j < k_j$.  
\end{prop} 
\vspace{-1mm}
\begin{proof} 
Put $h_j :=  l k_1 k_2 \cdots k_{j-1} - k_j$.  
Suppose that we have the equality $(\ast): (l-\alpha)\alpha + k_1^2 + k_2^2 + \cdots + k_j^2 - j = 
l k_1 \cdots k_j -1$ and the inequality $2 \leq k_1 \leq k_2 \leq \cdots \leq k_j$. 
Now we assume $l k_1 k_2 \cdots k_{j-1} - k_j \leq 0$. Then $l k_1 k_2 \cdots k_{j-1} k_j \leq k_j^2$. 
From $(\ast)$ we have $(l-\alpha)\alpha +1 + \sum_{i=1}^j k_i^2 - j -k_j \leq 0$. 
Thus 
\vspace{-2mm}
\begin{eqnarray*}
0 & \geq & (l-\alpha)\alpha + 1 - j + k_1^2 + \cdots k_{j-1}^2 \\
& \geq &  (l-\alpha)\alpha + 1 - j + 4(j-1) \\
& = &  3 j + (l-\alpha)\alpha -3 \geq  3j - 1.  
\end{eqnarray*} 
Since $j \geq 2$, the last expression is greater than or equal to $0$. 
This is a contradiction. Hence $0 < l k_1 k_2 \cdots  k_{j-1} - k_j$.  
 
Next we divide the proof into the two cases: $\alpha = 1$ and $\alpha \geq 2$.  
Firstly we consider the case $\alpha \geq 2$. Then $l$ must satisfy $l \geq 4$. 
Now we prove that if 
$\alpha \leq k_j \leq l 2^{j-1}-\alpha$, and $2 \leq k_i \leq l 2^{j-1}-\alpha$ for $1 \leq i\leq j-1$, then we have 
$\sum_{i=1}^{j} {k_i}^2 - l \prod_{i=1}^{j} {k_i} \leq 4(j-1) + \alpha^2 - l 2^{j-1} \alpha$. 
We prove this by using the idea and technique of the proof of \cite[Lemma 2 in p.42]{sato-kimura}.  
Put $b := l 2^{j-1}-\alpha$, and 
assume $\alpha \leq k_j \leq b$ and $2 \leq k_i \leq b$  for $1 \leq i\leq j-1$. 
We put $f(k_1, \ldots, k_j) := \sum_{i=1}^{j} {k_i}^2 - l \prod_{i=1}^{j} {k_i}$. 
For $1 \leq \mu \leq j-1$  we set 
\[M^\alpha_{\mu} := f(\underbrace{2,\ldots,2}_{\mu},\underbrace{b, \cdots, b}_{j-\mu-1}, \alpha) \ \mbox{and} \ 
M^b_{\mu} := f(\underbrace{2,\ldots,2}_{\mu},\underbrace{b, \cdots, b}_{j-\mu-1}, l 2^{j-1} - \alpha). \] 
Then we have $M^\alpha_{\mu} = 2^2 \mu +(j-\mu-1)b^2 + \alpha^2 - l 2^\mu b^{j-\mu-1} \alpha$, and  
$M^b_{\mu}$ $=$ $2^2 \mu +(j-\mu)b^2 - l 2^\mu b^{j-\mu}$. 
Since $M^\alpha_{j-1} - M^\alpha_{\mu} = (j-\mu-1)(2^2 -b^2) - l 2^\mu \alpha (2^{j-\mu-1} - b^{j-\mu-1})$, 
we obtain 
\begin{eqnarray*}
&& \frac{M^\alpha_{j-1} - M^\alpha_{\mu}}{b-2} \geq -(j-\mu -1)(b + 2) + l \alpha (j-\mu-1) 2^{j-2} \\ 
&=& (j- \mu -1)\{- b - 2 + l \alpha 2^{j-2} \} \geq 0.  
\end{eqnarray*}
On the other hand since 
$M^\alpha_\mu - M^b_{\mu} = \alpha^2 - b^2 + l 2^\mu b^{j-\mu-1} \{b - \alpha\}$, 
we obtain 
\begin{eqnarray*}
\frac{M^\alpha_\mu - M^b_{\mu}}{b-\alpha} = -(\alpha + b) + l 2^\mu b^{j-\mu-1}.    
\end{eqnarray*} 
When $\mu = j-1$, this value is equal to zero. Thus we have $M^\alpha_\mu - M^b_{\mu} \geq 0$. 
Since $f$ attains the maximum at the boundary points, we obtain the desired assertion $f \leq M^\alpha_{j-1}$.   

Now from the equality $(\ast)$ we have 
\vspace{-2mm}
\begin{eqnarray*}
0 &=& (l-\alpha)\alpha -(j-1) + k_1^2 + k_2^2 + \cdots + k_j^2  - l k_1 \cdots k_j \\ 
&\leq & (l-\alpha)\alpha -(j-1) + 4(j-1) + \alpha^2 - l 2^{j-1} \alpha \\
&=& (l-\alpha)\alpha  + 3(j-1) + \alpha(\alpha - l 2^{j-1}) \\
&=& (l-\alpha)\alpha  + 3(j-1) + \alpha(\alpha - l) - \alpha l (2^{j-1}-1) 
\end{eqnarray*}
Since $j \geq 2, l \geq 4, \alpha \geq 2$ the last expression is negative. 
However this is a contradiction. It follows that from the assumption $k_j \geq \alpha$ we obtain 
$k_j > l 2^{j-1} - \alpha$. 

From now on we consider the case $j =2$. 
Firstly we show that $k_2 > k_1 + \frac{l}{2}$. 
Assume that $k_2 \leq  k_1 + \frac{l}{2}$. 
Then we have 
\vspace{-1mm}
\begin{eqnarray*}
0 &=& (l-\alpha)\alpha -1 + k_1^2 + k_2^2 - l k_1 k_2 \\
&\leq & (l-\alpha)\alpha -1 + 2 k_2^2 - l(k_2 -\frac{l}{2})k_2. 
\end{eqnarray*} 
\vspace{-1mm}
Since $k_2 > 2 l -\alpha$, we have   
\vspace{-2mm} 
\begin{eqnarray*} 
0 &<& (l-\alpha)\alpha -1 + (2-l)(2 l-\alpha)^2 + \frac{l^2}{2}(2 l-\alpha) \\ 
&=& (l-\alpha)\alpha -1 + (2 l -\alpha)\{-\frac{3}{2}l^2 -2 \alpha  + l(4 + \alpha)\}. 
\end{eqnarray*} 
\vspace{-1mm} 
About the part of this expression  we have 
\[-\frac{3}{2}l^2 -2 \alpha  + l(4 + \alpha) = l(-\frac{3}{2}l +4) + \alpha (l-2).\]  
Since $\alpha \leq \frac{l}{2}$, the last expression is less than or equal to  
$l(-l +3) \leq -l$. 
It follows that $(l-\alpha)\alpha + (2 l -\alpha)(-l) <0$. 
This is a contradiction. Therefore $k_2 > k_1 + \frac{l}{2}$. 
 
Finally assume that $h_2 \geq  k_2$. This condition is equivalent to $l k_1 \geq 2 k_2$. 
Then from the equation $(\ast)$ we have 
\begin{eqnarray*} 
0 &=& (l-\alpha)\alpha -1 + k_1^2 + k_2^2 - l k_1 k_2 \\ 
&\leq & (l-\alpha)\alpha -1 + k_1^2 + k_2^2 - 2 k_2^2 \\
&=& (l-\alpha)\alpha -1 + k_1^2 - k_2^2 \\
&<& (l-\alpha)\alpha -1 + k_1^2 - (k_1 + \frac{l}{2})^2 \\[-1mm]
&\leq &  (l-\alpha)\alpha -1 - (l k_1 + \frac{l^2}{4}) \\[-1mm]
&\leq & \frac{l^2}{4} - \frac{l^2}{4} -1 - l k_1 < 0. 
\end{eqnarray*} 
This is a contradiction. Hence $h_2 < k_2$, which is our assertion in this Proposition. 

\smallskip
Next we consider the case $j \geq 3$. 
Now assume that $k_j \leq \frac{l}{4}(j-1)k_{j-1}$.   
Then we have 
\vspace{-2mm}
\begin{eqnarray*}
0 &=& (l-\alpha)\alpha  - (j-1)  + k_1^2 + \cdots + k_j^2 -l k_1 \cdots k_j \\
&\leq &  (l-\alpha)\alpha - (j-1) + k_1^2 + \cdots + k_j^2 - \frac{4}{j-1} k_1 \cdots k_{j-2} k_j^2. 
\end{eqnarray*}
\vspace{-1mm}
Since we have $\frac{4}{j-1} k_1 \cdots k_{j-2} \geq \frac{4}{j-1} 2^{j-2} \geq j+1$, this yields 
\vspace{-1mm}
\[(l-\alpha)\alpha - (j-1) + k_1^2 + \cdots + k_j^2 - \frac{4}{j-1} k_1 \cdots k_{j-2} k_j^2 \leq 
(l-\alpha)\alpha - (j-1) - k_j^2. \] 
\vspace{-1mm}
The inequality $k_j > l 2^{j-1} - \alpha$ gives 
\begin{eqnarray*}
0 &<& (l-\alpha)\alpha - (j-1) - (l 2^{j-1} - \alpha)^2 \\
& \leq & \frac{l^2}{4} - (j-1) - (3l)^2 < 0. 
\end{eqnarray*} 
This is a contradiction.  Hence we obtain  $k_j > \frac{l}{4}(j-1)k_{j-1}$. 
 
Finally suppose that $h_j \geq  k_j$. This condition is equivalent to 
$k_j \leq \frac{l}{2}k_1 \cdots k_{j-1}$.  
Then combining this assumption with the equation $(\ast)$ yields 
\begin{eqnarray*}
0 &=& (l-\alpha)\alpha  - (j-1)  + k_1^2 + \cdots + k_j^2 -l k_1 \cdots k_j \\
&\leq &  (l-\alpha)\alpha - (j-1) + k_1^2 + \cdots + k_j^2 - 2 k_j^2 \\
&<& (l-\alpha)\alpha - (j-1) + k_1^2 + \cdots + k_{j-1}^2 - (\frac{l}{4})^2 (j-1)^2 k_{j-1}^2 \\
&\leq & (l-\alpha)\alpha - (j-1) -(j-1)(\frac{l^2}{4^2}(j-1) - 1)k_{j-1}^2.  
\end{eqnarray*}
Since $j \geq  3$ and $k_{j-1} \geq 2$, the last expression is less than or equal to 
\begin{eqnarray*} 
&& (l-\alpha)\alpha - (j-1) - (l^2 - 8) \\ 
&=& (l - \alpha)\alpha - (j-1) - (l^2 - 8) \\
&\leq &  (l - \alpha)\alpha - l^2 + 6  \\
&=& -l(l - \alpha) - \alpha^2 + 6 < 0.  
\end{eqnarray*} 
This is a contradiction, which concludes $h_j <  k_j$. 

Now we consider the case $\alpha =1$. In this case we have to also consider the case $l=3$. 
We can prove the inequality $h_j < k_j$ by almost the same way as the case $\alpha \geq 2$. 
In the following we give the outline of the proof with emphasizing the difference between the cases 
$\alpha =1$ and $\alpha \geq 2$.  

Firstly if we assume $2 \leq k_i \leq l 2^{j-1} -2$ for $1 \leq i \leq j$, then from Lemma 2 of 
\cite[p.~42]{sato-kimura} 
we can directly obtain $\sum_{i=1}^{j} {k_i}^2 - l \prod_{i=1}^{j} {k_i}^2 \leq 4 j - l 2^j$. 
Combining this inequality with the equation $(\ast)$ implies a contradiction by a similar argument to the case 
$\alpha \geq 2$. Hence we obtain $k_j > l 2^{j-1} -2$. Now we divide the proof into the two cases $j=2$ 
and $j\geq 3$. When $j = 2$, by using the inequality $k_j > 2 l -2$  
we can prove $k_2 > k_1 + \frac{l}{3}$ similarly to the case $\alpha \geq 2$, but not $k_2 > k_1 + \frac{l}{2}$.  
Moreover by using $k_2 > k_1 + \frac{l}{3}$, we can obtain $h_2 < k_2$. 
When $j \geq 3$, by using the inequality $k_j > l 2^{j-1} -2$, we can obtain $k_j > \frac{l}{4}(j-1)k_{j-1}$. 
Moreover if we suppose $h_j \geq k_j$, then 
combining $k_j > \frac{l}{4}(j-1)k_{j-1}$ 
with 
the equation $(\ast)$ yields a contradiction. Thus we obtain $h_j < k_j$.  
\end{proof} 
The proof of Proposition \ref{fund prop of cast 2} is a generalization of the one of 
\cite[Lemma 7.3]{hkato}. 
By Proposition \ref{fund prop of cast 1} and \ref{fund prop of cast 2} we obtain the following: 

\medskip
\begin{thm}\label{determanaiton of successive castling transform} 
Let $l$ and $\alpha$ be positive natural numbers such that $l \geq 3$ and $\alpha \leq l - \alpha$. 
Let $\theta = (k_1, k_2, \ldots, k_j)$ be a set of positive natural numbers. 
Then $\theta$ is obtained by a finite number of 
castling transformations from $\alpha$ if and only if 
$\theta$ gives a solution of the equation 
$(\ast)$:   
$(l-\alpha)\alpha  -(j-1) + k_1^2 +  \cdots + k_j^2 - l k_1 \cdots k_j = 0$ and satisfies 
$k_i \geq  \alpha$ for $1 \leq i \leq j$.  
\end{thm} 

\medskip
Furthermore Proposition \ref{fund prop of cast 2} implies a stronger result: suppose that a set of  positive natural numbers 
$\theta$ $=$ ($k_1$, $k_2$, $\ldots$, $k_j$) gives a solution of 
the equation $(\ast)$ and $\theta$ is not contained in the cube 
$C_\alpha^j = \{(x_1 ,\ldots, x_j) \mid |x_i| \leq \alpha - 1\}$.  Then $\theta$ is obtained by a finite number of 
castling transformations from the solution $\alpha$ of $(\ast)$ with $j=1$. 
Thus we obtain the following.

\medskip
\begin{prop} 
Assume that there exists no positive integer solutions $(k_1, \ldots$, $k_j)$  of the equation $(\ast)$ 
satisfying $k_i \leq  \alpha - 1$. 
Then any positive integer solution of $(\ast)$ is obtained by a finite number of castling 
transformations from the solution $\alpha$. 
\end{prop} 

From many computations of the equation $(\ast)$, it seems that there exists no positive integer solution 
$\theta = (k_1, k_2, \ldots, k_j)$ such that $\theta$ is contained in the cube $C_\alpha^j$. 
Hence we conjecture the following. 

\medskip
\begin{conj} 
There exists no positive integer solutions $(k_1, \ldots, k_j)$ of the equation $(\ast)$ 
satisfying $k_i \leq  \alpha - 1$ for $1 \leq  i \leq j$. 
\end{conj} 

\medskip
This conjecture is true for $\alpha = 1,2,3$ or $j = 1$. 

\medskip 
Now let $(Q, \omega)$ be a Grassmannian Cartan connection of type $(\beta, \alpha)$ over a manifold $M$, where we 
assume $l = \alpha + \beta \geq  3$ and $\alpha \leq \beta$.  
Denote by $S(l, \alpha)$ the set of positive natural number solutions of the equation ($\ast$).   
We consider the tree $T$ of Cartan connections obtained by  
successive castling transformations from $(Q, \omega)$. 
Each node is written as 
$(Q \times \prod_{i=1}^j PL(k_i), \omega' \times \prod_{i=1}^j (\Lambda_1)_i)$, where $\omega' = \omega$ or 
$*\omega$ and  $(\Lambda_1)_i = \Lambda_1$ or $(\Lambda_1)_i = *\Lambda_1$. 
The model space of $(Q \times \prod_{i=}^j PL(k_i), \omega' \times \prod_{i=1}^j (\Lambda_1)_i)$ is 
$PL(l) \times \prod_{i=1}^j PL(k_i)/PL(l) \times \prod_{i=1}^j PL(k_i)_{w}$, where $w$ is a point of 
$P(\R^l \otimes \bigotimes_{i=1}^j \R^{k_i})$. 
We define the map $\Phi: T \to S(l, \alpha)$ by 
$(Q \times \prod_{i=1}^j PL(k_i), \omega' \times \prod_{i=1}^j (\Lambda_1)_i) \mapsto (k_1, \ldots, k_j)$. 
Each node of $T$ induces a base space $M_{k_1 \times \cdots \times k_j}$ by Proposition 
\ref{base space of successive Cartan conne}. 
Thus we obtain the map $\overline{\Phi}$ from the set of the base spaces of nodes in $T$  
to $S(l, \alpha)$ defined by 
$\overline{\Phi}: M_{k_1 \times \cdots \times k_j} \mapsto (k_1, \ldots, k_j)$. 
Moreover $\overline{\Phi}$ induces the map $\Psi$ from the set of the fibers 
$\prod_{i=1}^j PL(k_i)$ of base spaces of nodes in $T$ to 
$S(l, \alpha)$.  The map $\Psi$ is bijective from Theorem \ref{determanaiton of successive castling transform}.   
Thus we obtain the following.   

\medskip 
\begin{thm}\label{lastthm}
There is a one-to-one correspondence between the set of 
structure groups $\prod_{i=1}^j PL(k_i)$ of the base spaces obtained by a finite number of castling transformations from $(Q, \omega)$  
and the set of solutions 
$(k_1, \ldots, k_j)$  of  the equation 
\[(\ast) \ \alpha \beta + k_1^2 +  \cdots + k_j^2 - (j-1) - (\alpha + \beta) k_1 \cdots k_j = 0.\]   
satisfying $k_i \geq \alpha$ \mbox{\rm(}$1\leq  i \leq j$\mbox{\rm)} and  $j\geq 1$. 

Each solution $(k_1, \ldots, k_j)$ corresponds to a manifold 
equipped with a projective structure, which is projectively flat if 
$(Q, \omega)$ is flat.  
\end{thm} 

\medskip
From this theorem we obtain Theorem \ref{cor of thm}.  

\medskip
\begin{rem} 
For a principal fiber bundle $N$ over $M$ corresponding to the solution $(k_1, \ldots, k_j)$, 
a bundle $L= N/{PL(k_j)}$ is equipped with a projective structure again if 
$l k_1 \cdots  k_{j-1} - k_j = 1$ and this manifold corresponds to a solution $(k_1, \cdots, k_{j-1})$ of 
the equation $(\ast)$: $\alpha \beta -(j-2) + k_1^2 + \cdots + k_{j-1}^2 - l k_1 \cdots k_{j-1} = 0$. 
Indeed the set of numbers $(k_1, \ldots, k_{j-1}, l k_1 \cdots k_{j-1} - k_j)$ also gives a solution of $(\ast)$.  
If we have $l k_1 \cdots k_{j-1} - k_j = 1$, then $(k_1, \cdots, k_{j-1}, 1)$ is a solution of $(\ast)$. 
Thus $(k_1, \cdots, k_{j-1})$ is a solution of $(\ast)$.  
If we have $l k_1 \cdots  k_{j-1} - k_j \neq 1$, then $L$ admits a Grassmannian structure of type  
$(l k_1 \cdots  k_{j-1} - k_j, k_j)$, whose corresponding Grassmannian Cartan connection is flat if 
$(Q, \omega)$ is flat.  
\end{rem}
\medskip
\begin{rem} 
The equation ($\ast$) $\alpha \beta + k_1^2 +  \cdots + k_j^2 - (j-1) - (\alpha + \beta) k_1 \cdots k_j  = 0$ is 
a generalization of the equation ($\ast \ast$) 
$a^2 + k_1^2 + \cdots k_j^2 - j - 2 a k_1 \cdots k_j = 0$ ($a = 2$, $3$ or $5$) 
which we obtained in our preceding paper \cite[Theorem1.1]{hkato}.  
Indeed put $\alpha := a-1$ and  $\beta := a + 1$ in ($\ast$). Then we obtain ($\ast \ast$).     
\end{rem}

\section{Description of flat projective structures} 
Let $\{(U_\alpha, \varphi_\alpha)\}_{\alpha \in A}$ be a flat Grassmannian structure on $M$, and 
$(Q, \omega)$ be the corresponding flat Grassmannian Cartan connection on $M$. 
Then by a finite number of castling transformations, we obtain a manifold $N$ 
corresponding to the solution  $(k_1, \ldots, k_j)$ in Theorem \ref{lastthm}. 
The manifold $N$ is equipped with a Cartan connection $(Q \times \prod_{i=1}^j PL(k_i), \omega' \times \prod_{i=1}^j (\Lambda_1)_i)$ of type 
$PL(l) \times \prod_{i=1}^j PL(k_i)/PL(l) \times \prod_{i=1}^j PL(k_i)_w$, where $w$ is an element of 
$V_{k_j, l k_1 \cdots k_{j-1}}$ $\subset$ $P(\R^l \otimes \bigotimes_{i=1}^j \R^{k_i})$ obtained in the process 
of successive castling transformations.      
Then $N$ admits a flat projective structure.   
We also showed that $N$ is a principal fiber bundle over $M$ with group 
$\prod_{i=1}^j PL(k_i)$. 
Furthermore in Proposition \ref{base space of successive Cartan conne} 
we described the relation of the base spaces corresponding to the solutions of $(\ast)$.  
 
Now finally we shall describe the flat projective structure on $N$ by using the flat Grassmannian structure 
$\{(U_\alpha, \varphi_\alpha)\}_{\alpha \in A}$ on $M$.  
For each connected component $C$ of the nonempty intersection $U_\alpha \cap U_\beta$, 
the coordinate change $\varphi_\beta \circ \varphi_\alpha^{-1}|_C$ is given by an element 
$\tau(C; \beta, \alpha)$ of $PL(l)$. 
Let $\pi_Q: Q \to M$ and $\pi: PL(l) \to Gr_{m,l}$ be the projections.
Denote by $\widetilde{U}_\alpha$ the open subset $\pi^{-1}(\phi_\alpha(U_\alpha))$ of $PL(l)$.    
Then $\widetilde{U}_\alpha$ is naturally regarded as a principal fiber bundle over $U_\alpha$ and we 
denote the projection $\widetilde{U}_\alpha \to U_\alpha$ by $\pi_\alpha$. 
Put $\widetilde{U}_\alpha' := \widetilde{U}_\alpha$, $\pi_\alpha' := \pi_\alpha$ and 
$\tau(C; \beta, \alpha)' = \tau(C; \beta, \alpha)$ 
if $\omega' = \omega$ or 
$\widetilde{U}_\alpha' := *\widetilde{U}_\alpha$, $\pi_\alpha' := \pi_\alpha \circ *$ and 
$\tau(C; \beta, \alpha)' = *\tau(C; \beta, \alpha)$ if $\omega' = *\omega$.  

\medskip 
\begin{thm}\label{proj sr on N}
The manifold $N$ is diffeomorphic to a patchwork of the open submanifolds $\widetilde{U}'_\alpha \otimes \bigotimes_{i=1}^j PL(k_i).w$  of the projective space $P(\R^l \otimes \bigotimes_{i=1}^j \R^{k_i})$: 
\vspace{-2.5mm}
\[N \simeq \bigsqcup_{\alpha \in A}\widetilde{U}'_\alpha \otimes \bigotimes_{i=1}^j PL(k_i).w/_\sim. \vspace{-2mm} \]   
The elements $\tilde{g}_{\alpha} = g \otimes A_1 \otimes \cdots \otimes A_j.w$ and 
$\tilde{h}_{\beta} = h \otimes B_1 \otimes \cdots \otimes B_j.w$ of the open submanifolds 
$\widetilde{U}'_\gamma \otimes \bigotimes_{i=1}^j PL(k_i).w$ \mbox{\rm(}$\gamma = \alpha$, $\beta$\mbox{\rm)} 
are identified iff 
\vspace{0mm}
\begin{eqnarray*}
&\mbox{\rm (i)}& \ \pi_\alpha'(g) = \pi_\beta'(h), \\ 
&\mbox{\rm (ii)}& \  \tilde{h}_\beta = 
\tau(C; \beta, \alpha)' \otimes id \otimes \cdots \otimes id \cdot \tilde{g}_\alpha, \ \ \ \ \ 
(\pi_\alpha'(g) \in C).   
\vspace{-1mm}
\end{eqnarray*} 
Thus $N$ admits an atlas inducing a flat projective structure, whose coordinate changes are the same as ones of the flat Grassmannian structure on $M$. 
\end{thm} 
\begin{proof} 
Let $\pi_Q: Q \to M$ and $\pi: PL(l) \to Gr_{m,l}$ be the projections.  We denote by $\omega_G$ the 
Maurer-Cartan form of $PL(l)$.  
Since the Cartan connection $(Q, \omega)$ is constructed from the atlas 
$\{(U_\alpha, \varphi_\alpha)\}_{\alpha \in A}$, 
the Cartan connection $(\pi_Q^{-1}(U_\alpha), \omega)$ is isomorphic to 
the one $(\widetilde{U}_\alpha, \omega_G)$ via an isomorphism 
$\widetilde{\varphi}_\alpha$. 
Let $\widetilde{\varphi}_\beta$ be an isomorphism between $(\pi_Q^{-1}(U_\beta), \omega)$ and 
$(\widetilde{U}_\beta, \omega)$. 
Then the transition function between $\widetilde{\varphi}_\alpha$ and $\widetilde{\varphi}_\beta$ over a 
connected component $C$ of $U_\alpha \cap U_\beta$ is  
given as follows: for an element $z$ of $\pi_Q^{-1}(C)$, 
$\widetilde{\varphi}_\beta(z) = \tau(C; \beta, \alpha) \widetilde{\varphi}_\alpha(z)$.  
  
The manifold $\pi_Q^{-1}(U_\alpha) \times \prod_{i=1}^j PL(k_i)$ itself is an open submanifold of 
$Q \times \prod_{i=1}^j PL(k_i)$ and moreover is the bundle over the open submanifold $\pi_N^{-1}(U_\alpha)$ of $N$, where $\pi_N$ is the projection $N \to M$.  
we define a map $\psi_\alpha: \pi_Q^{-1}(U_\alpha) \times \prod_{i=1}^j PL(k_i) \to PL(l) \times \prod_{i=1}^j PL(k_i)$ by 
$(z, A_1, \ldots, A_j) \mapsto (\widetilde{\varphi}_\alpha{}'(z), A_1', \ldots, A_j')$, 
where $\widetilde{\varphi}_\alpha{}'(z) := \widetilde{\varphi}_\alpha(z)$ and $A_i' = A_i$ if 
$\omega' = \omega$ 
or $\widetilde{\varphi}_\alpha{}'(z) := *\widetilde{\varphi}_\alpha(z)$ and $A_i' = *A_i$ 
if $\omega' = *\omega$. 
Then since ${\widetilde{\varphi}_\alpha}^*\omega_G = \omega_\alpha$,  
the pullback of the Maurer-Cartan form of $PL(l) \times \prod_{i=1}^j PL(k_i)$ by $\psi_\alpha$ 
is equal to 
$\omega'_\alpha \times \prod_{i=1}^j (\Lambda_1)_i$. 
The map $\psi_\alpha$ is compatible with the action $PL(l) \times \prod_{i=1}^j PL(k_i)_w$. 
Since the homogeneous space $PL(l) \times \prod_{i=1}^j PL(k_i)/PL(l) \times \prod_{i=1}^j PL(k_i)_w$  
is identified with the orbit $PL(l) \otimes \bigotimes_{i=1}^j PL(k_i).w$,  
$\psi_\alpha$ induces a map $\bar{\psi}_\alpha$ of the base spaces, which is a diffeomorphism between 
the open submanifold $\pi_N^{-1}(U_\alpha)$ of $N$ and the open submanifold $\widetilde{U}'_\alpha \otimes \bigotimes_{i=1}^j PL(k_i).w$ 
of $P(\R^l \otimes \bigotimes_{i=1}^j \R^{k_i})$. 

On the intersection $\pi_Q^{-1}(C) \times \prod_{i=1}^j PL(k_i)$ the transition function between 
$\psi_\alpha$ and $\psi_\beta$ is given as follows: 
$\psi_\beta \circ \psi_\alpha^{-1} = \tau(C; \beta, \alpha)' \times id \times \cdots \times id$. 
Thus concerning the base spaces the coordinate change 
$\bar{\psi}_\beta \circ \bar{\psi}_\alpha^{-1}$ over 
$\pi_N^{-1}(C)$ is given by $\tau(C; \beta, \alpha)' \otimes id \otimes \cdots \otimes id$.   
Hence the atlas $\{(\pi_N^{-1}(U_\alpha), \bar{\psi}_\alpha)\}_{\alpha \in A}$ of $N$, 
which is induced form the Cartan connection 
$(Q \times \prod_{i=1}^j PL(k_i), \omega' \times \prod_{i=1}^j (\Lambda_1)_i)$, 
naturally gives an atlas of a flat projective structure on $N$ 
via the inclusion $PL(l) \times \prod_{i=1}^j PL(k_i) \to PL(\R^l \otimes \bigotimes_{i=1}^j \R^{k_i})$.   
\end{proof} 

We constructed a flat projective structure $\mathscr{A}$ on $N$ in Theorem \ref{proj sr on N}, which is induced from 
the flat Cartan connection $(Q \times \prod_{i=1}^j PL(k_i), \omega' \times \prod_{i=1}^j (\Lambda_1)_i)$. 
On the other hand $(Q \times \prod_{i=1}^j PL(k_i), \omega' \times \prod_{i=1}^j (\Lambda_1)_i)$ induces 
a flat projective Cartan connection $(P, \xi)$ by Proposition \ref{inducing cartan connection}.  
The Cartan connection $(P, \xi)$ induces a flat projective structure $\mathscr{B}$, which is the same as 
the one stated in Theorem \ref{lastthm} and thus Theorem \ref{cor of thm}. We say that two atlases are equivalent if they are compatible. Finally we shall prove the following:  
\medskip  
\begin{prop} 
The flat projective structures $\mathscr{A}$ and $\mathscr{B}$ on $N$ are equivalent.  
\end{prop} 
\begin{proof} 
Firstly we note that the model space 
$PL(l) \times \prod_{i=1}^j PL(k_i)/PL(l) \times \prod_{i=1}^j PL(k_i)_w$ 
is a subgeometry of 
$PL(\R^l \otimes \bigotimes_{i=1}^j \R^{k_i})/PL(\R^l \otimes \bigotimes_{i=1}^j \R^{k_i})_w$. 
We now generally consider a homogeneous space 
$A/B$ which is a subgeometry of $A'/B'$.  Thus there is the inclusion $F$ of $A$ into $A'$ and 
$F$ induces a diffeomorphism $\hat{F}$ of $A/B$ onto an open subset of $A'/B'$. 
Denote by $\omega_A$ the Maurer-Cartan form of $A$, then $(\pi:A \to A/B, \omega_A)$ gives the standard flat Cartan connection on $A/B$. Likewise we obtain the flat Cartan connection $(\pi':A' \to A'/B', \omega_{A'})$. 
Then we have $dF \circ \omega_A = F^* \omega_{A'}$. 
Let $(Q, \omega)$ (resp. $(Q', \omega')$) be a flat Cartan connection of type $A/B$ (resp. $A'/B'$) on a manifold $N$, and 
assume that $(Q, \omega)$ is a subgeometry of $(Q', \omega')$ defined by a bundle homomorphism 
$\tilde{F}: Q \to Q'$.  We consider an atlas $\mathscr{C}$ of 
$(A, A/B)$-structure on $N$ induced from $(Q, \omega)$ and an atlas $\mathscr{D}$ of $(A', A'/B')$-structure on $N$ induced from $(Q', \omega')$ (See \cite{hkato} for the terminology).  
Hence any chart of $\mathscr{C}$ is constructed as follows: 
For arbitrary point $p \in N$,  
there exists a neighbourhood $U$ of $p$ in $N$ and a bundle isomorphism $f: \pi_Q^{-1}(U) \to V$, where 
$V$ is an open subset of $A$, such that $f^* \omega_A = \omega$. Then $f$ induces a diffeomorphism 
$\bar{f}: U \to \pi(V)$ and we obtain the chart $(U, \bar{f})$ belonging to $\mathscr{C}$. 
Likewise there exists an isomorphism $f': \pi_{Q'}^{-1}(U') \to V'$ around $p$,  where $V'$ is an open subset of $A'$, and $f'$ induces a diffeomorphism $\bar{f'}: U' \to \pi'(V')$. 
Note that $\hat{F} \circ \bar{f}$ gives a diffeomorphism of $U$ onto 
the open subset $\hat{F} \circ \pi(V)$ of $A'/B'$. 
Thus we obtain two charts $(U, \hat{F} \circ \bar{f})$ and $(U', \bar{f'})$ of type $A'/B'$.  

Now we compare the composite $f' \circ \widetilde{F} \circ f^{-1}$ and $F$. 
Since $(Q, \omega)$ is a subgeometry of $(Q', \omega')$, we have the equality 
\begin{equation}\label{eq-form}
(f' \circ \widetilde{F} \circ f^{-1})^* \omega_{A'} = dF \circ \omega_A = F^* \omega_{A'}.
\end{equation}  
Let $C$ be a connected component of $U \cap U'$.  Then $\bar{f}(C)$ is a connected open subset of $A/B$. 
We consider the inverse image $\pi^{-1}(\bar{f}(C))$ and decompose it into the connected components 
$\{D_\gamma \}_{\gamma \in \Gamma}$. 
Since 
there exists a connection in the principal fiber bundle $\pi:A \to A/B$,  
it can be shown that each $D_\gamma$ is mapped onto $\bar{f}(C)$ by $\pi$. 
Moreover from the equality (\ref{eq-form}) there exists a unique element $a_\gamma$ of $A'$ for each $\gamma$  
such that $f' \circ \widetilde{F} \circ f^{-1} = a_\gamma \cdot F$ on $D_\gamma$ 
(cf. Theorem 1.2.4 of \cite{cap-slovak}). 
We can prove that 
all the elements of $\{a_\gamma\}_{\gamma \in \Gamma}$ are the same, and we put $a' := a_\gamma$.  
Thus on $\pi^{-1}(\bar{f}(C))$ we have $f' \circ \widetilde{F} \circ f^{-1} = a' \cdot F$. 
Hence concerning two coordinates $\hat{F} \circ \bar{f}$ and $\bar{f'}$ of type $A'/B'$,  
the coordinate change 
$\bar{f'} \circ (\hat{F} \circ \bar{f})^{-1}$ on $C$ is given by the translation of $a'$. 
Let $\mathscr{C}'$ be an atlas of $(A', A'/B')$-structure on $N$ given rise to by $\mathscr{C}$ via $F$.   
Then from the above discussion it follows that $\mathscr{C}'$ is equivalent to $\mathscr{D}$.  
  
By applying this result to the case that $PL(l) \times \prod_{i=1}^j PL(k_i)/PL(l) \times \prod_{i=1}^j PL(k_i)_w$ 
is a subgeometry of   
$PL(\R^l \otimes \bigotimes_{i=1}^j \R^{k_i})/PL(\R^l \otimes \bigotimes_{i=1}^j \R^{k_i})_w$,  
we obtain the assertion of this proposition.  
\end{proof}

\medskip 
\begin{center}
\bf Acknowledgments
\end{center}

The author expresses his 
gratitude to Thomas Bruun Madsen for his encouragement. 
Thanks are due to the referee for the helpful comments, especially Theorem \ref{proj sr on N} 
was made from the referee's suggestion.  
The author wishes to thank the Osaka City University Advanced Mathematical Institute and 
the King's College London for financial support and hospitality. 


\begin{small}
  \parindent0pt\parskip\baselineskip

  Hironao Kato \\
  Osaka City University Advanced Mathematical Institute, 3-3-138 Sugimoto, Sumiyoshi-ku,  
  Osaka 558-8585 Japan \\
  Department of Mathematics, King's College London, 
  Strand, London WC2R 2LS, United Kingdom. 
  
  \textit{E-mail}: katohiro@sci.osaka-cu.ac.jp
\end{small}
\end{document}